\documentclass[12pt]{amsart}
\usepackage{geometry}                
\usepackage{graphicx}
\usepackage{amssymb,amsmath}
\usepackage{mathtools}
\usepackage[pagebackref=true, colorlinks=true]{hyperref}
\usepackage{float,subfigure}
\usepackage{tikz}
\usepackage{xfrac}
\usepackage{comment}
\DeclareMathOperator{\Res}{Res}

\theoremstyle{plain} 
\newtheorem{theorem}{Theorem}[section]
\newtheorem{lemma}[theorem]{Lemma}

\newtheorem{proposition}[theorem]{Proposition}
\theoremstyle{definition} 
\newtheorem{definition}[theorem]{Definition}
\newtheorem{remark}[theorem]{Remark}

\newcommand{\R}{\mathbb{R}}
\newcommand{\C}{\mathbb{C}}
\newcommand{\N}{\mathbb{N}}

\newcommand{\re}{\operatorname{Re}}
\newcommand{\im}{\operatorname{Im}}
\newcommand*\conj[1]{\overline{#1}}

\begin{document}
		
		\begin{title}
			{Dihedralization of Minimal Surfaces in $\mathbb{R}^3$}
		\end{title}
		
		\author{Ramazan Yol}
		
		\address{Ramazan Yol\\Department of Mathematics\\Indiana University\\
			Bloomington, IN 47405
			\\USA}
		
		\date{}
	
		\begin{abstract}
	It is a well known phenomenon that many classical minimal surfaces in Euclidean space also exist with “higher dihedral symmetry”.
More precisely, these surfaces are solutions to free boundary problems in a wedge bounded by two vertical planes with varying angle.
We study the limit of such surfaces when the angle converges to 0. In many cases, these limits are simpler than the original surface, and can be used in conjunction with the implicit function theorem to give new existence proofs of the original surfaces with small dihedral angle.
This approach has led to the discovery of new minimal surfaces as well.

		\end{abstract}
		
		\maketitle	
\section{Introduction}\label{sec:introduction}
It is well-known that minimal surfaces with arbitrary genus and high rotational symmetry group exist. Famous  examples, such as the Scherk surface and the Costa surface, often have ``companion" surfaces with $n$-fold rotational symmetry. In order to study a minimal surface $X$ with $n$-th order rotational symmetry group $\langle r\rangle$, it suffices to study a fundamental piece $W$ of $X/\langle r\rangle$. In full generality, $W$ can be seen as a minimal surface in a ``twisted wedge" which is omitted here for the sake of simplicity. In this paper we will only consider the surfaces with high dihedral symmetries. Consequently, $W$ can be seen as a solution to a free boundary problem in a wedge bounded by two vertical planes that meet at an angle $\frac{2\pi}{n}$. Visa-versa a given minimal surface $W$, whose boundary lies on two symmetry planes meeting at an angle  $\frac{2\pi}{n}$ (thus meets the bounding planes orthogonally), can be extended by reflections to a complete minimal surface in $\R^3$. We can also generalize the free boundary problems by allowing for any non-negative real wedge angle $2\pi\alpha$ where $\alpha\geq0$. This generalization allow us to consider a continuous family of solutions. Therefore, one can investigate minimal surfaces with high rotational symmetry of any order by studying free boundary problems in a wedge bounded by two vertical planes with a varying angle.

This perspective raises many questions: Is it possible to solve the free boundary problem for every  $2\pi\alpha$? If that is not the case, for what upper and lower bounds of  $\alpha$, can we solve it? What would the solutions look like for the limiting case $\alpha=0$? Clearly, our understanding of the moduli spaces of rotationally symmetric minimal surfaces would be enhanced if we had the answers to these questions.

For the purpose of understanding moduli spaces of minimal surfaces, a key method, frequently used in the past, has been to work on limits of families of minimal surfaces. For instance, Traizet showed in a series of articles  \cite{tr1,tr2a,tr2} that new examples of minimal surfaces in Euclidean space can be regenerated by utilizing noded Riemann surface limits. In this paper, inspired by regeneration methods, we regenerate solutions to certain free boundary problems by utilizing the limiting solution as $\alpha$ goes to $0$. Note that, for  $\alpha=0$, the bounding planes of the wedges become parallel, and with appropriate scaling, we obtain a free boundary problem in a slab bounded by two parallel, vertical planes. As $\alpha$ approaches $0$, there is of course a change in the edges of wedges since the wedges in the free boundary problem transform into a slab in the limit case. Fig \ref{fig:ScherkTower} illustrates the phases of this phenomenon; note, in particular, the dramatic change of the edges as we transition from Fig \ref{fig:ScherkTower}(b) to (c). 
\begin{figure}[h]
	\begin{center}
		\subfigure[Karcher’s Scherk Towers]{\includegraphics[width=1.5in]{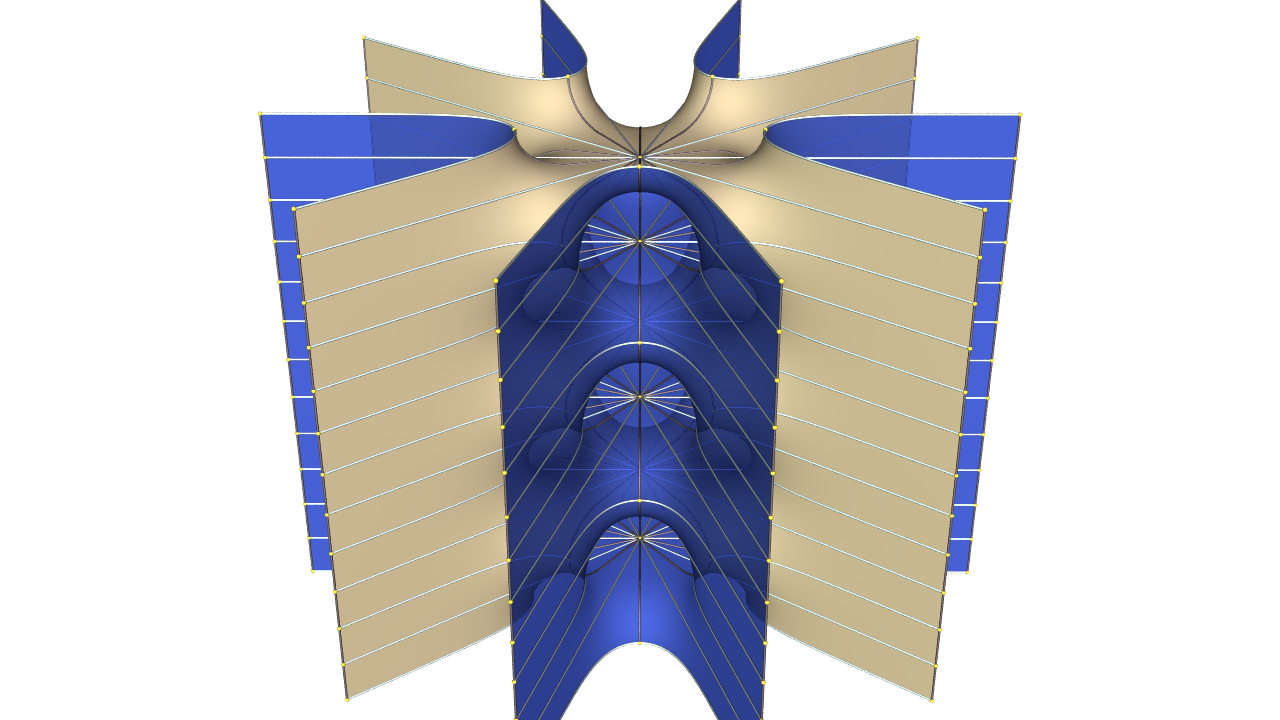}}
		\subfigure[A Minimal Wedge of Scherk Tower]{\includegraphics[width=1.5in]{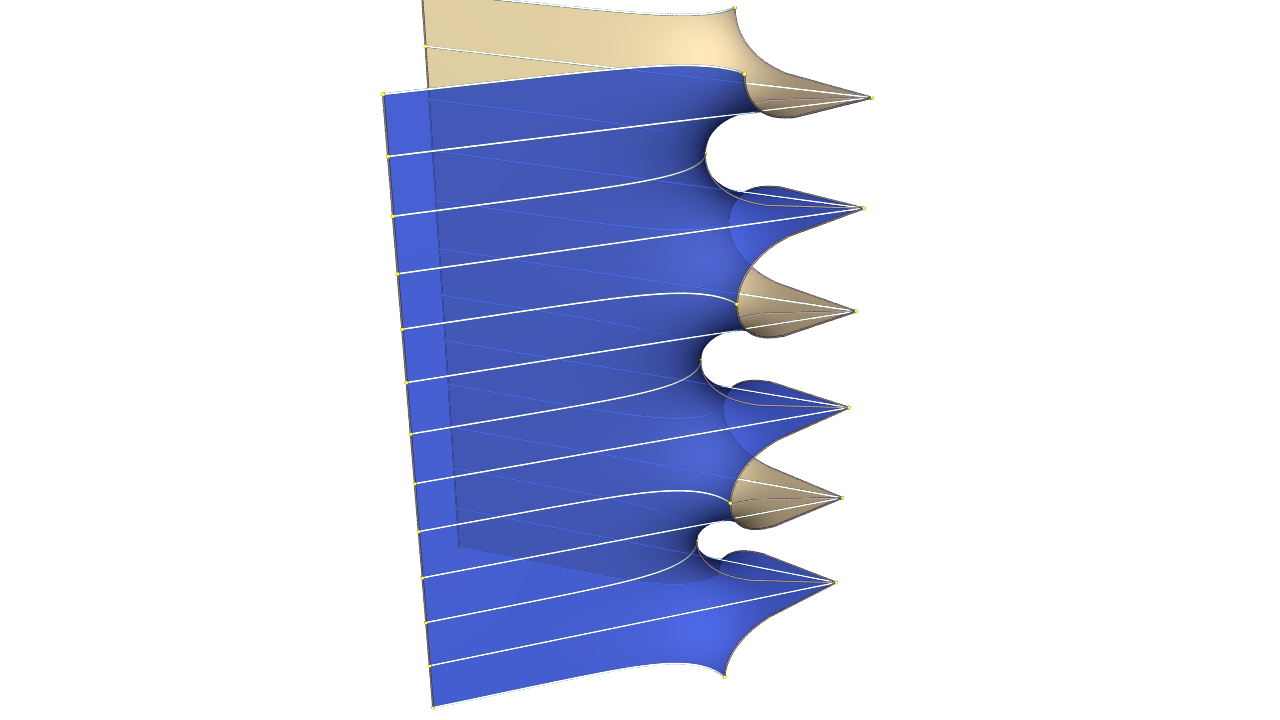}}
		\subfigure[Limit of a Wedge]{\includegraphics[width=1.5in]{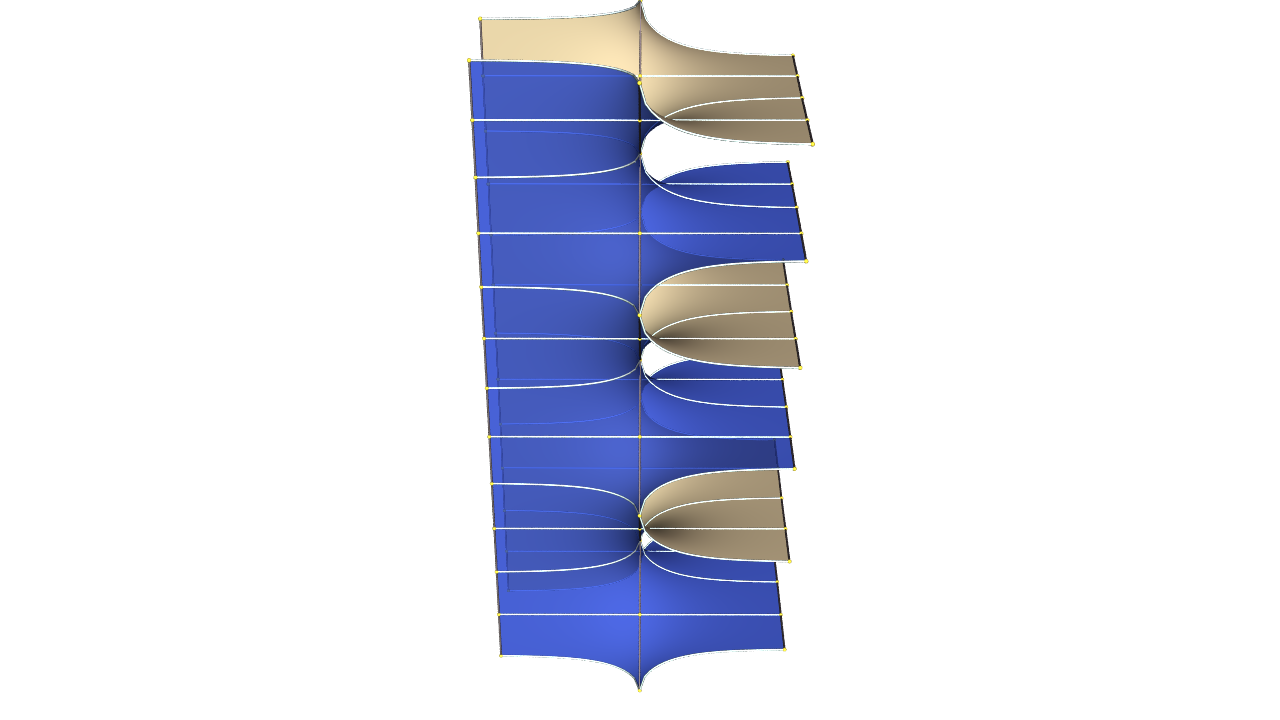}}
		\subfigure[Extended Limit]{\includegraphics[width=1.5in]{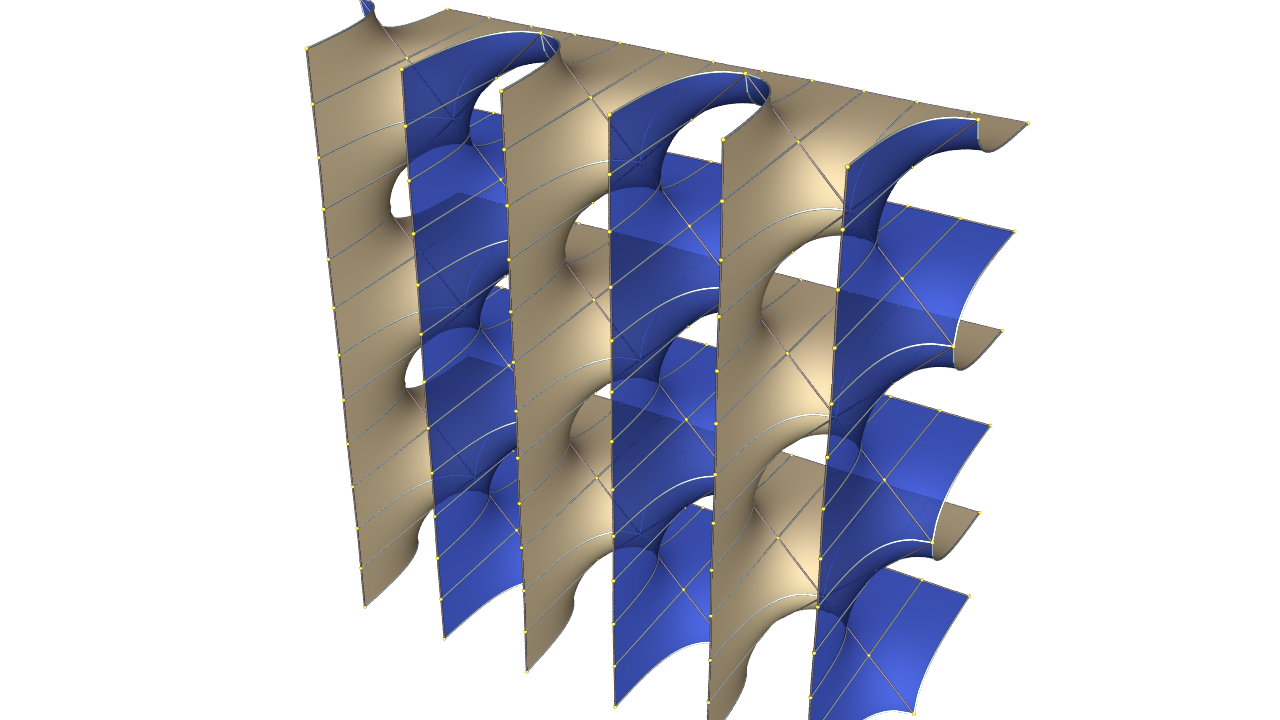}}
	\end{center}
	\caption{Dihedralization of Scherk Tower}
	\label{fig:ScherkTower}
\end{figure} 

This work presents results for minimal surfaces in the following class.
\begin{definition}
	A minimal surface $X_n:M_n\longrightarrow\R^3$ is called $n$-dihedral if $D_n$ is the largest dihedral subgroup of Sym$(X_n)$. When $X_n$ is a variant of an existing minimal surface $X$, we refer to $X_n$ as the dihedralized $X$.
\end{definition} 

A very well known example of $n$-dihedral surfaces are the so-called Scherk Towers found by Hermann Karcher (see Fig. \ref{fig:ScherkTower}(a) and \cite{ka4}). In order to put the discussion  of minimal surfaces with high dihedral symmetry into context of  free boundary problems in wedges, we define:
\begin{definition}\label{def:wedge}
	A minimal surface  $W_{2 \pi \alpha}:\mathcal{M}_{\alpha}\longrightarrow\R^3$, whose boundary lies on two symmetry planes making an angle $2 \pi\alpha \geq0$, is called a minimal wedge. In particular, the surface meets the bounding planes orthogonally.
\end{definition}
Hereafter, we will call $\alpha$ the dihedral angle on account of the fact that an $n$-dihedral minimal surface $X_n$ gives rise to a minimal wedge $W_{\frac{2\pi}{n}}$ which can be defined as a fundamental piece of ${X_n}/\langle r\rangle$. For example, the minimal wedge corresponding to a Scherk Tower of $2n$ ends is given in Fig. \ref{fig:ScherkTower}(b)). As mentioned earlier, for any $\alpha$, $W_{2\pi\alpha}$ can be extended to a complete minimal surface in $\R^3$, and we will refer to the extension of $W_{0}$ as the dihedral limit of the sequence $X_n$. As $\alpha$ goes to $0$, the limit $W_{0}$ becomes a minimal surface in a slab. Hence, the rotational symmetries of the sequence $X_n$ are replaced by the translational periodicity of its dihedral limit.

Surprisingly, we cannot always expect a dihedralized version of a minimal surface to exist. For example, our numerical experiments show that the $n$-dihedral variants of Wohlgemuth's  surface of genus $3$ (see \cite{wo3}) do not exist for integers $n \geq 3$. On the other hand, the dihedralized variations of the Costa-Hoffman-Meeks (see \cite{HK97}) and the Callahan-Hoffman-Meeks (see \cite{chm2}) surfaces are known to exist for all $n$; and they converge to the singly periodic Scherk and the Karcher-Meeks-Rosenberg surfaces (see \cite{ka4}, \cite{mr3}) respectively. It is therefore interesting to classify minimal surfaces on the basis of their ``dihedralizability".

This paper reveals that, indeed, there are a number of families of minimal surfaces that can be regenerated from their dihedral limit. In many cases, these limits are simpler than the original surfaces, and can be used in conjunction with the implicit function theorem to obtain new and simpler existence proofs of the original surfaces with small dihedral angle. This approach, which we call dihedralization, has also led us to the discovery of new minimal surfaces.

A prototype example of regeneration with dihedralization is an existence proof for the dihedralized Chen-Gackstatter surface of genus $3(n-1)$ (which we denote by $DE_{3,n}$) for large enough positive integer $n$, obtained by first proving the existence of its limit surface. Generally, this requires one to solve a $2$ dimensional system of equations that include integrand terms of exponents $1/(n-1)$. It is known that such surfaces exist (see \cite{tha3}, \cite{ww1}), but our argument provides a shorter proof that does not involve integration of the abovementioned terms with complicated exponents. Furthermore, through dihedralization arguments, we introduce two novel surfaces illustrated in Figures \ref{fig:DCCW}(a) and \ref{fig:DKS}(a), as described by the following two theorems.

\begin{theorem}\label{DCCWtheorem}
		For sufficiently large $n\in \N$, there exists an $n$-dihedral, finite type minimal surface $DCCW_{n}$ of genus $2n-2$ with four catenoidal ends, which is invariant under a reflection with respect to a horizontal plane. Moreover, as $n \rightarrow \infty$, $DCCW_{n}$ converges to an embedded singly periodic minimal surface of genus $0$ with $6$ annular ends, which is invariant under horizontal translations. The limit surface is a less symmetric Scherk Tower of $6$ ends and it is also invariant under a vertical reflection and a horizontal reflection.              
\end{theorem} 
The surface $DCCW_{n}$ as it appears in Theorem \ref{DCCWtheorem}, is a variant of Wolgemuth's Surface of genus $2$ (see \cite{wo3}) with all $4$ catenoidal ends instead of planar ends. Hence we call it the ``dihedralized catenoidal Costa-Wohlgemuth surface".
\begin{figure}[H]
	\begin{center}
		\subfigure[$\alpha>0$]{\includegraphics[width=3in]{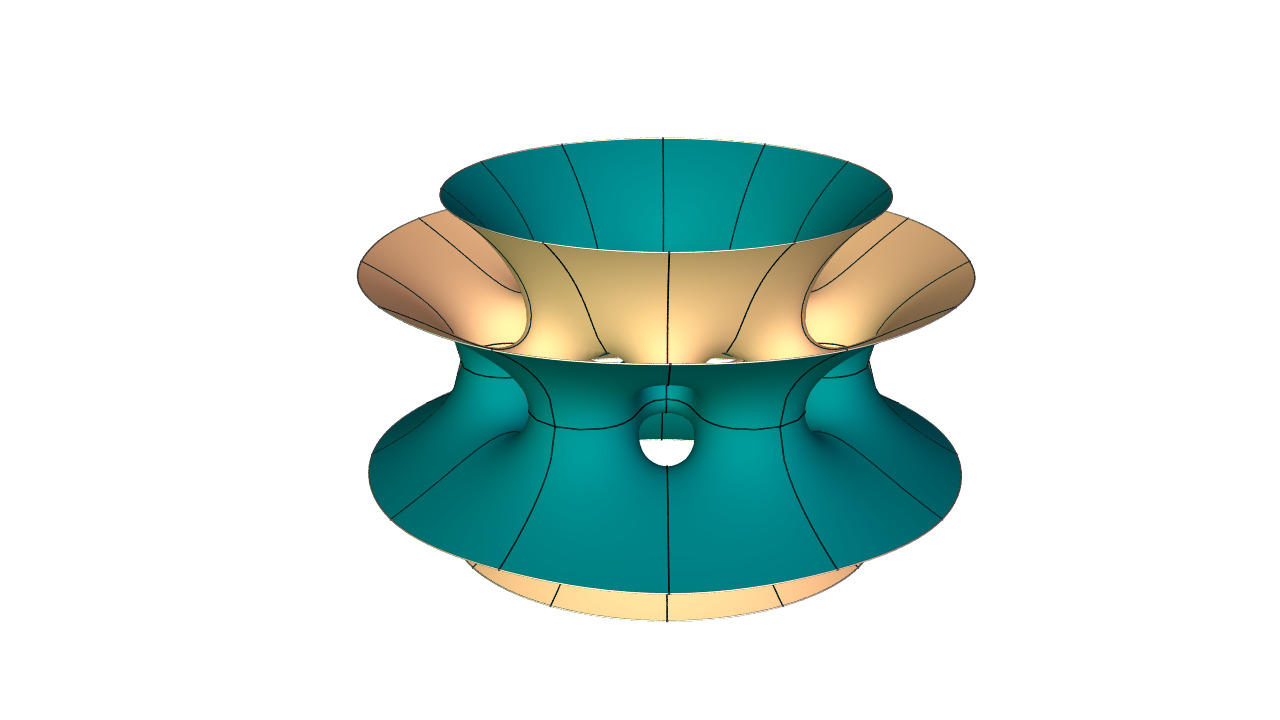}}
		\subfigure[$\alpha=0$]{\includegraphics[width=2.6in]{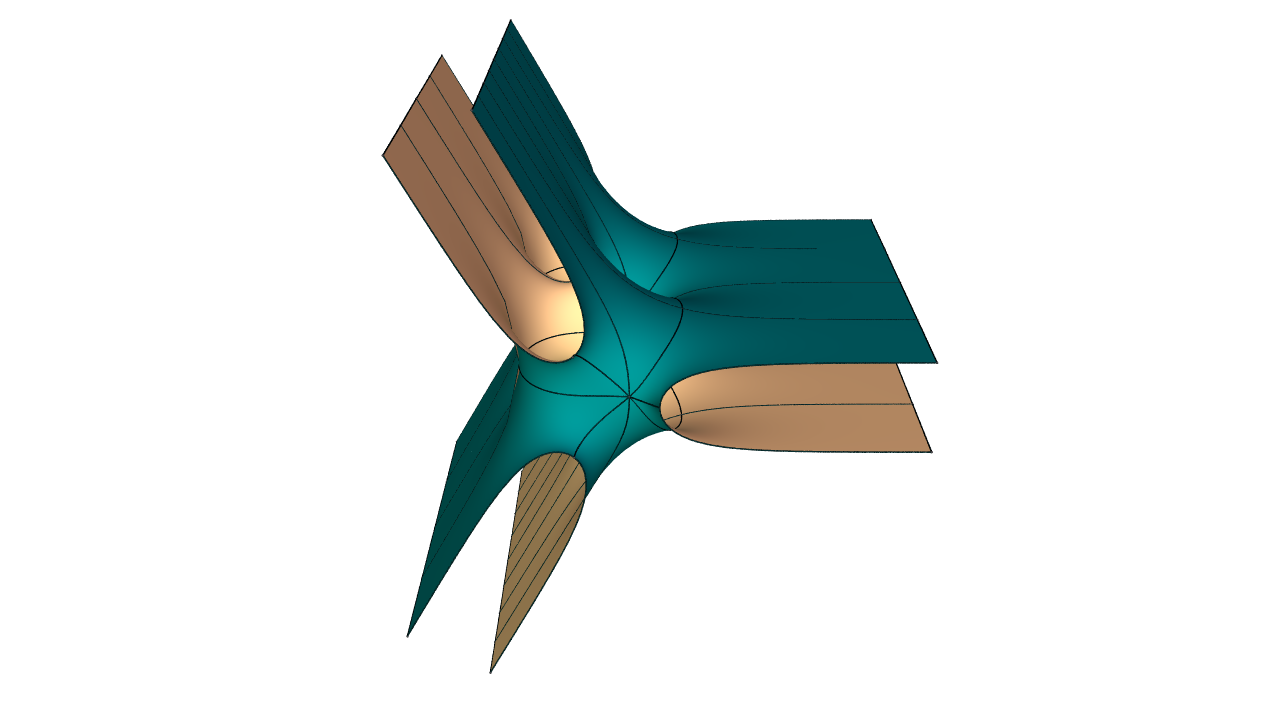}}
	\end{center}
	\caption{$DCCW_n$ and its limit $\R^3$}
	\label{fig:DCCW}
\end{figure}
	\begin{theorem}\label{DKStheorem}
	For sufficiently large $n\in \N$, there exists a complete, embedded, singly periodic $n$-dihedral minimal surface $DKS_{n}$ of genus $n$ with $2n$ annular ends, which is invariant under reflection with respect to a horizontal plane. Moreover, as $n \rightarrow \infty$, $DKS_{n}$ converges to an embedded doubly periodic minimal surface of genus $1$ with $4$ annular ends. The limit surface, known as the doubly periodic Karcher-Scherk surface of genus $1$, is invariant under a vertical reflection and a horizontal reflection.
\end{theorem}
\begin{figure}[H]
	\begin{center}
		\subfigure[$\alpha>0$]{\includegraphics[width=3in]{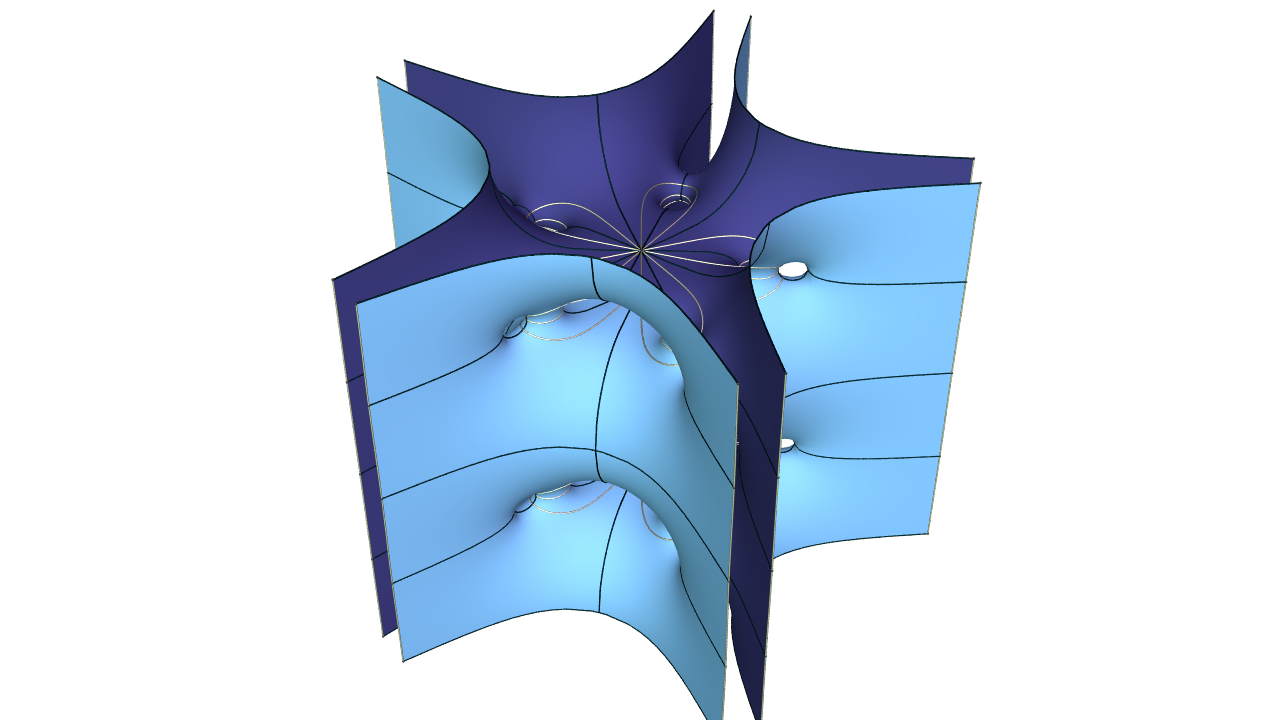}}
		\subfigure[$\alpha=0$]{\includegraphics[width=2.6in]{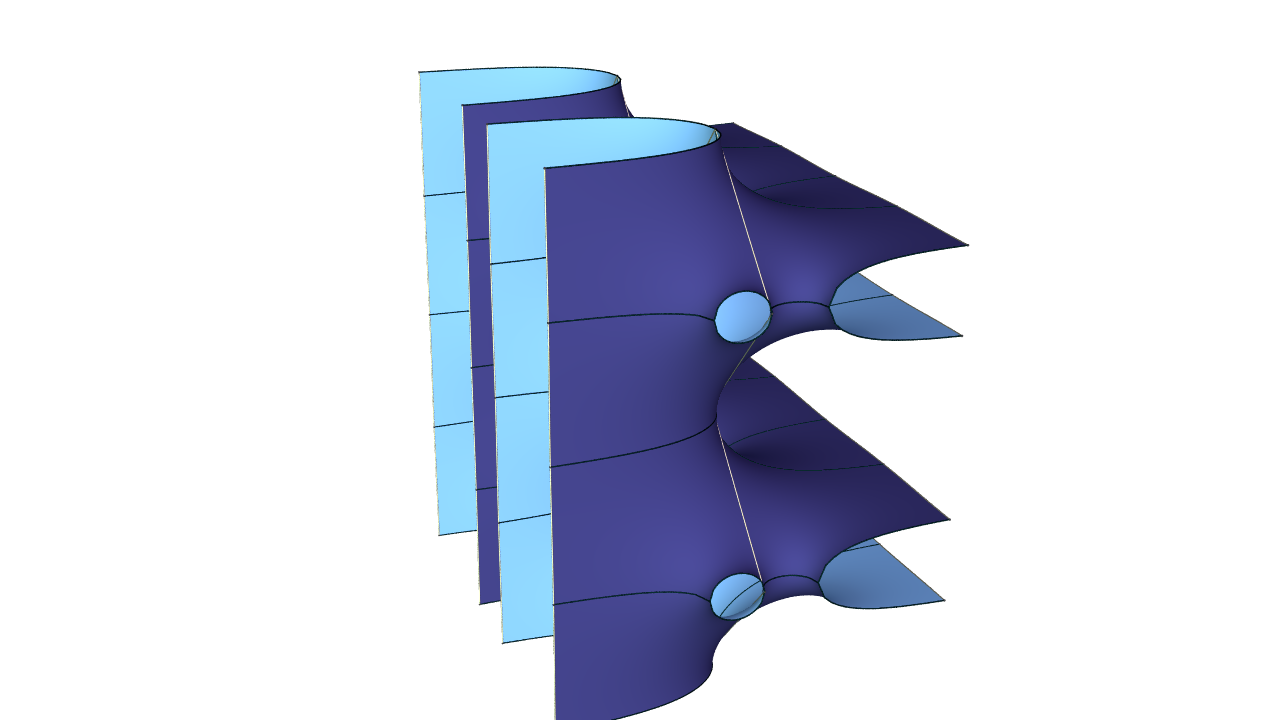}}
	\end{center}
	\caption{DKS$_n$ and its limit $\R^3$}
	\label{fig:DKS}
\end{figure}
On the other hand, the surface given in Theorem \ref{DKStheorem}, and aptly called the ``dihedralized Karcher Scherk surfaces with handles", can be seen as Scherk towers with added handles.\\

The paper is organized as follows: In Section \ref{sec:geom}, we demonstrate that the assumed symmetries of our surfaces have strong implications on the flat structures of the Weierstrass $1$-forms, allowing us to parametrize the surfaces via Schwarz-Christoffel maps and $\vartheta$-functions which are also introduced in Section \ref{sec:geom}. Section \ref{sec:DE} contains an existence proof for the $DE_{3,n}$ surfaces and the corresponding wedges $W_{2\pi\alpha}$. This proof serves as the prototype example of the new dihedralization method. In Section \ref{sec:DWW},  employing dihedralization, we provide a concise proof for the existence of a $n$-dihedral Weber-Wolf surface with 4 catenoidal ends and a planar end. While acknowledging the existing knowledge about these surfaces, we present brief proofs that highlight the methodologies of the dihedralization process.

Building upon the methodology established in Section \ref{sec:DE} and \ref{sec:DWW}, we extend our exploration to novel surfaces. In Section \ref{sec:DCCW}, we give a proof for Theorem \ref{DCCWtheorem}. Finally, in Section \ref{sec:DKS}, we establish the proof of Theorem \ref{DKStheorem} by leveraging $\vartheta$-functions and Schwarz-Christoffel maps.
\section{Geometry of the Weierstrass Representation}\label{sec:geom}

Let a minimal map (i.e. a conformal parametrization of a minimal surface) be given by

\begin{equation}\label{f}
f(z) = \re \int^z (\omega_1, \omega_2, \omega_3)
\end{equation}
where

$$ \omega_1 = {}  \frac12 (\frac1G-G)\, dh, \qquad \omega_2 = {}  \frac{i}2 (\frac1G+G)\, dh, \qquad \omega_3= {} dh.$$

Here, the meromorphic function $G$ is the stereographic projection of the Gauss map, and the holomorphic 1-form $dh$ is called the height differential.

Recall that multiplying $dh$ by a real factor scales the surface, and multiplying it by $e^{i t}$ is the Bonnet deformation. 
Multiplying $G$ by a real factor is called the L{\'o}pez-Ros deformation, while multiplying $G$ by $e^{i t}$ rotates the surface about a vertical axis by the angle $\varphi$.


Let $f:U \to \R^3$ be a minimal map, given by Weierstrass data $G$ and $dh$.
Introduce  $\Omega_k(z) =\int^z \omega_k$,  $\Phi_1(z) =\int^z G\, dh$, and $\Phi_2(z) =\int^z \frac1G\, dh$.

We will next explain that the particular symmetries we assume about our surfaces imply that the flat structures of $dh$ , $G\, dh$ and $\frac1G\, dh$ satisfy certain conditions. This is crucial for our line of reasoning, because it will allow us to define  $dh$ , $G\, dh$ and $\frac1G\, dh$ either as integrands of Schwarz-Christoffel maps from the upper half plane to Euclidean polygons or as integrands of rational functions with $\vartheta$-function factors on tori.

\begin{proposition}\label{prop:symmetries} Symmetries of minimal surfaces. 
	
	\begin{enumerate}
		\item Suppose that $\Omega_3$, $\Phi_1$, and $\Phi_2$ extend continuously to  the real interval $(a,b)\subset \partial U$ and map it to a segment orthogonal to, to a segment making angle $\alpha$ with, and to a segment making angle $-\alpha$ with the real axis, respectively. Then the Schwarz reflection principle guarantees that $\Omega_3$, $\Phi_1$, and $\Phi_2$ and thus $f$ can be extended across $(a,b)$ by reflection. We claim that this extension of $f$ is realized by a $180^\circ$ rotation  about  $f(a,b)$, which is a horizontal straight line in $\R^3$ making angle $\alpha$ with the $x$-direction.\\
		
		\item Suppose that $\Omega_3$, $\Phi_1$, and $\Phi_2$ extend continuously to  the real interval $(a,b)\subset \partial U$ and map it to a segment parallel to, to a segment making angle $\alpha$ with, and to a segment making angle $-\alpha$ with the real axis, respectively. Then the Schwarz reflection principle guarantees that $\Omega_3$, $\Phi_1$, and $\Phi_2$ and thus $f$ can be extended across $(a,b)$ by reflection. Then $f(a,b)$ is a reflectional symmetry curve in a vertical plane  in $\R^3$ making angle $\alpha$ with the $x$-axis.\\
		
		\item Suppose that $\sigma$ is a conformal involution of $U$ such that $\sigma^*G=\conj{\frac{1}{G}}$ and $\sigma^*dh=-\conj{dh}$, then $f$ is symmetric with respect to reflection on a horizontal plane in $\R^3$ that contains image of fixed points of $\sigma$ under $f$.
	\end{enumerate}

\end{proposition}
\begin{proof}
	To see (1) and (2), we first note that we can assume that $\alpha=0$. Otherwise we multiply $G$ by $e^{-\alpha i}$: This rotates the segments $\Phi_1(a,b)$ and $\Phi_2(a,b)$ to become parallel to the real axis, and leaves $\Omega_3(a,b)$ unchanged. On the other hand, it rotates the surface about a vertical  axis by angle $-\alpha$.
	
	Since now $\alpha=0$, extension across $(a,b)$  conjugates $G\, dh$  and $\frac1G\, dh$. Consequently,
	this leaves $\re\omega_1$ unchanged, while it turns   $\re\omega_2$  and $\re\omega_3$ into $-\re\omega_2$ and $-\re\omega_3$.
	
	Vice versa, if a minimal surface contains a horizontal straight line (necessarily a symmetry line) that is parametrized by a segment $(a,b)\subset \R$, the Weierstrass integrals above map $(a,b)$ to segments with the appropriate angles.
	
	Similarly,    $\Omega_3$, $\Phi_1$, and $\Phi_2$ map the real interval $(a,b)\subset U$ to a  segment parallel to, a segment making angle $\alpha$ with, and a segment making angle $-\alpha$ with the real axis, respectively, if and only if $f(a,b)$ is a reflectional symmetry curve in a vertical plane making angle $\alpha$ with the $x$-direction.
	
	Finally, let $\sigma$ be as in the proposition, then it fixes $\re\omega_{1}$ and $\re\omega_{2}$ while it turns $\re\omega_{3}$ to $-\re\omega_{3}$.  As a result $\sigma$ is realized by a reflection on a horizontal plane that contains fixed points of $\sigma$.

\end{proof}
In order to construct Weierstrass Data on tori, we will use the will use $\vartheta$-functions given below. For our purposes, $\vartheta$-functions on tori are analogous to linear functions on $\C$.
\begin{equation*}
	\vartheta(z)=\vartheta(z,\tau)=\sum_{n=-\infty}^{\infty}e^{\pi i(n+\frac{1}{2})^2\tau+2\pi i(n+\frac{1}{2})(z-\frac{1}{2})}
\end{equation*}
is one the classical Jacobi $\vartheta$-functions and it is an entire function with simple zeroes at the lattice points of the integer lattice spanned by $1$ and $\tau$, moreover, in a fundamental parallelogram, $\vartheta(z)$ has no further zeroes. it also enjoys the following properties:
\begin{enumerate}
	\item $\vartheta(-z)=-\vartheta(z)$,
	\item $\vartheta(z+1)=-\vartheta(z)$,
	\item $\vartheta(z+\tau)=-e^{-\pi i\tau-2\pi iz }\vartheta(z)$,
	\item $\vartheta'(0)\neq0$,
	\item $\vartheta(\conj{z})=\conj{\vartheta(z)}$ for $\re(\tau)=0$,
\end{enumerate}
Note that property (5) implies that $\vartheta(z)$ is real on the real line when $\tau$ is purely imaginary.\\

Our method for constructing minimal surfaces begins with the initial creation of a simply connected fundamental piece of minimal wedges $W_{2pi\alpha}$. These initial components serve as the foundation, and then we leverage prescribed symmetries to extend them into complete surfaces. In the next section, we will not only illustrate the construction methodology but also present a proof for the existence of the Chen-Gackstatter Surface of genus $3(n-1)$ with dihedral symmetry. This proof will be established through a dihedralization argument.
 
\section{The Chen-Gackstatter Surface of genus $3(n-1)$ with dihedral symmetry}\label{sec:DE}

Our objective is to prove the following theorem, providing insights into the methodology and details of dihedralization along the way. This will serve as a template for future examples. 
\begin{theorem}\label{DEtheorem}
	For sufficiently large values of $n\in \N$, there exists a finite type $n$-dihedral minimal surface DE$_{3,n}$ of genus $3n-3$ with one Enneper type end which is invariant under $180^\circ$ horizontal rotations. Moreover, as $n \rightarrow \infty$, DE$_{3,n}$ converges to a singly periodic minimal surface of genus 0 with an Enneper type end and, the dihedral limit surface is invariant under vertical translations and $180^\circ$ rotations around two horizontal lines.       
\end{theorem}
	\begin{figure}[H]
	\begin{center}
		\includegraphics[width=6cm]{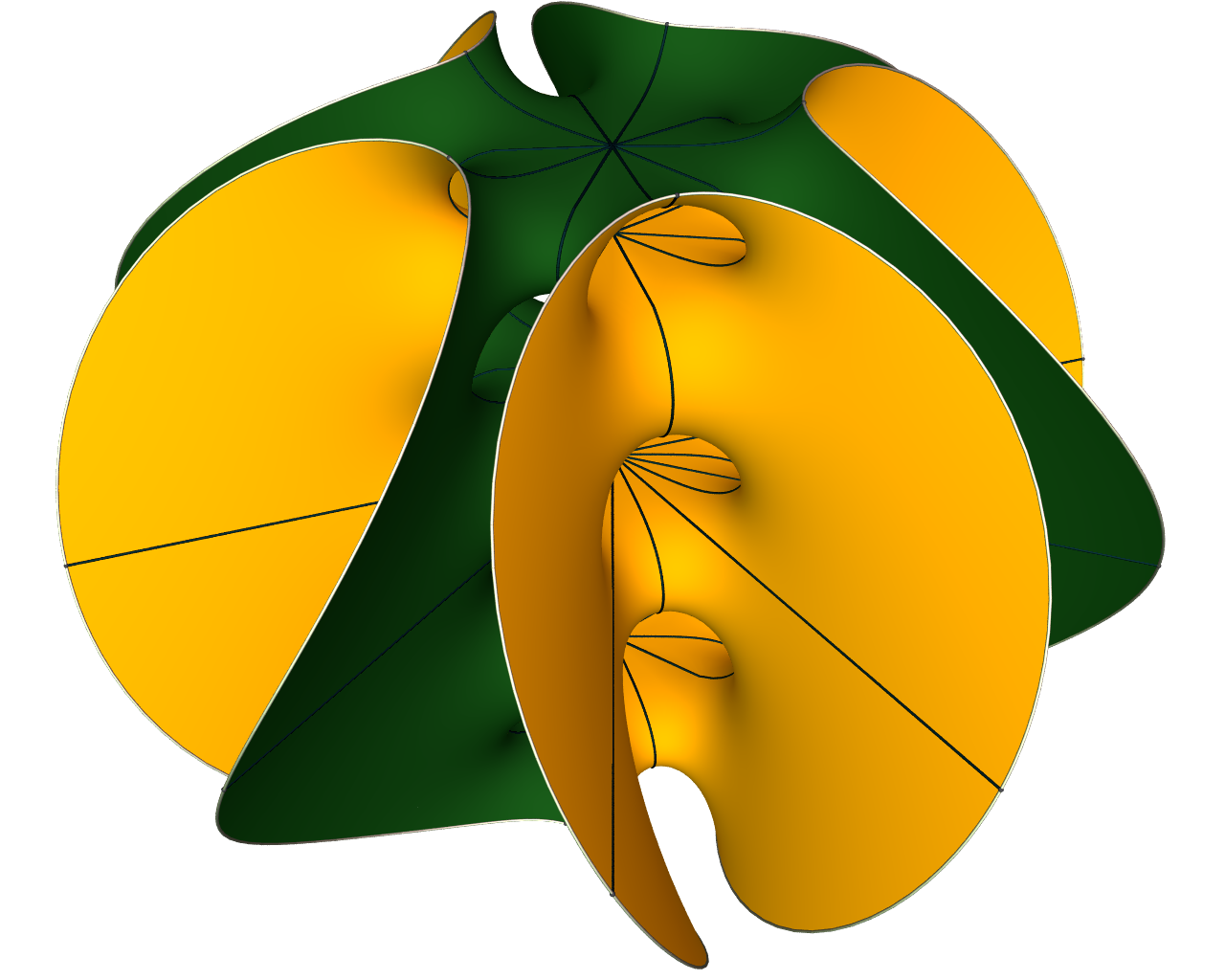}
	\end{center}
	\caption{Dihedralized Chen-Gackstatter Surface, $\alpha=\sfrac{1}{5}$}
	\label{fig:DEfull}	
\end{figure}
\begin{proof}[Proof of Theorem \ref{DEtheorem}]
We begin the proof by giving Lemma \ref{DEoctagon}, which contains a construction of a half of a wedge, corresponding to $DE_{3,n}$, which can be seen as a minimal octagon in Fig \ref{fig:DEoctagon}. For $\alpha=\frac{1}{n}$, this minimal octagon corresponds to $\frac{1}{2n}$th of the surface $DE_{3,n}$. 
\begin{lemma}\label{DEoctagon}
	For any $0<\rho$, $0\leq\alpha$ and $0<1<a<b<\infty$, there is a Weierstrass map $f$, cf. Section \ref{sec:geom}, that maps the upper half plane to a minimal octagon with one vertex at $\infty$ in $\R^3$. The edges of this octagon lie in vertical symmetry planes either parallel to the $x$-axis or make an angle $\alpha\pi$ with the $x$-axis alternatingly (if $\alpha=0$, then all the symmetry planes are parallel to the $x$-axis and all the vertices of the minimal polygons are at $\infty$). Moreover, there is a half straight line on this octagon that passes from two of the vertices and divides the octogon into two symmetrical pieces.  
\end{lemma}
\begin{figure}[H]
	\begin{center}
		\includegraphics[width=6cm]{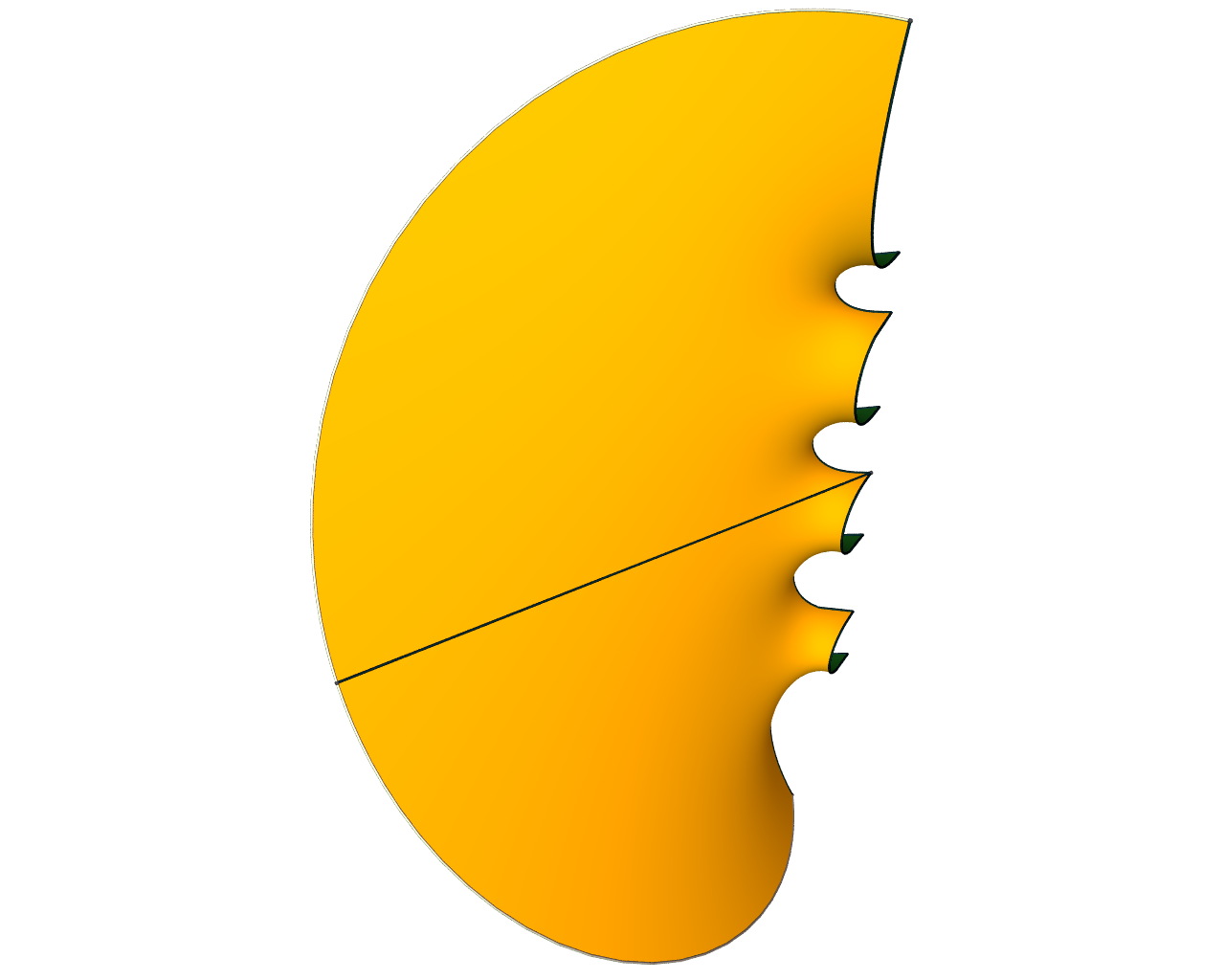}
	\end{center}
	\caption{A minimal octagon in $\R^3$ corresponding to $DE_{3,n}$}
	\label{fig:DEoctagon}
\end{figure}
\begin{proof}[Proof of Lemma \ref{DEoctagon}] First, when $\alpha$ is nonzero, we utilize two Schwarz-Christoffel maps to construct such a minimal octagon. We define,
			\begin{equation}\begin{split}	\label{deintegrands}
			\varphi_{1}&=
			z^{1-\alpha}
			(1-z^2)^{\alpha-1}
			(a^2-z^2)^{1-\alpha}
			(b^2-z^2)^{\alpha-1}dz ,
			\\	
			\varphi_{2}&=
			z^{\alpha-1}
			(1-z^2)^{1-\alpha}
			(a^2-z^2)^{\alpha-1}
			(b^2-z^2)^{1-\alpha}dz.
	\end{split}	\end{equation} 
One of the two Schwarz-Christoffel maps is defined by integrand $\varphi_{1}$ given in \eqref{deintegrands}, and this map sends the upper half plane to left side of the polygon Figure \ref{fig:DEflat}(a). The other Schwarz-Christoffel map is defined by integrand $\varphi_{2}$ given in \eqref{deintegrands} and sends the upper half plane to right side of the polygon seen in Figure \ref{fig:DEflat}(b).
\begin{figure}[H]
	\begin{center}
		\subfigure[$Gdh$]{\includegraphics[width=2in]{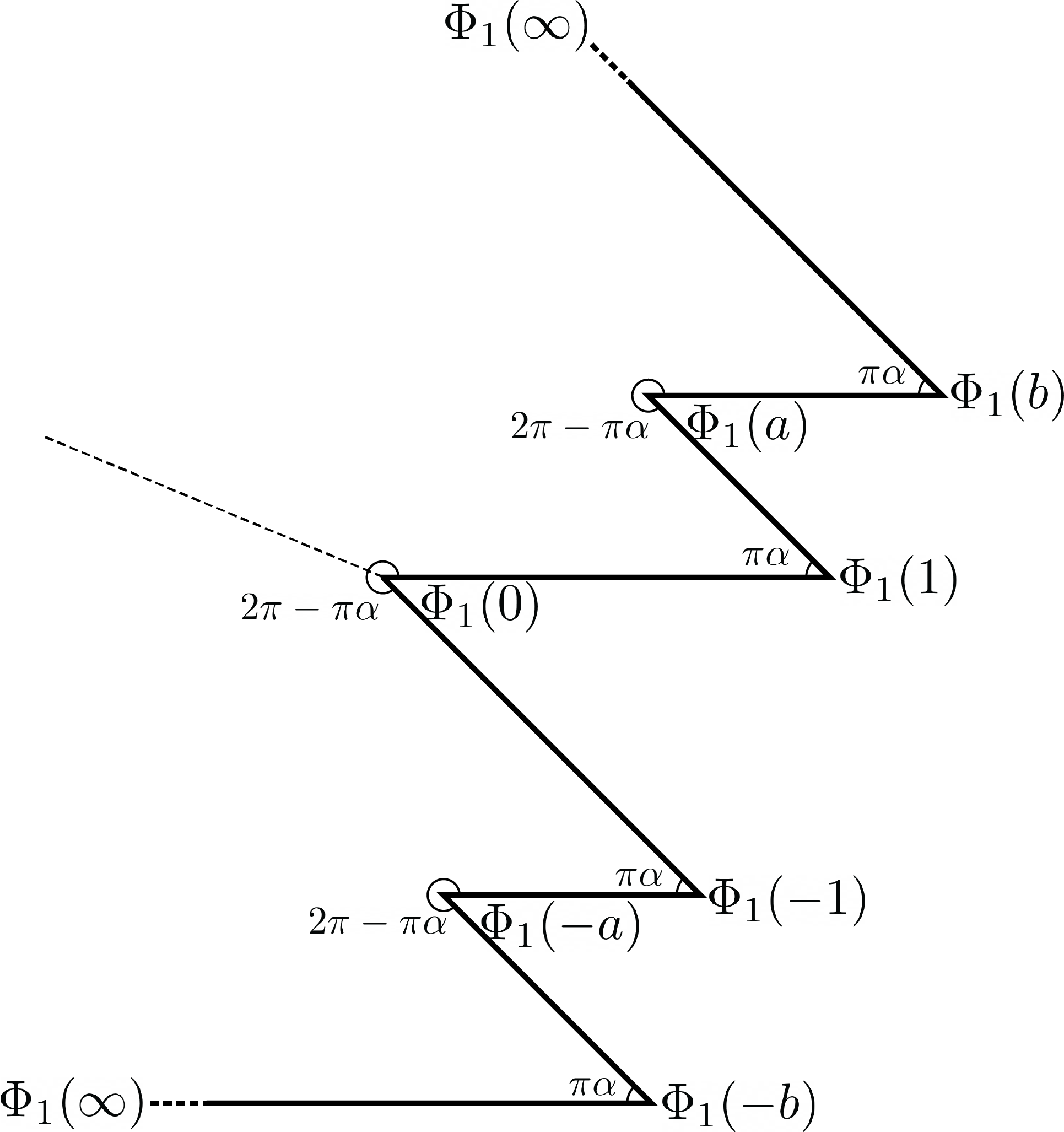}}
		\qquad\qquad\qquad
		\subfigure[$\frac{1}{G}dh$]{\includegraphics[width=2in]{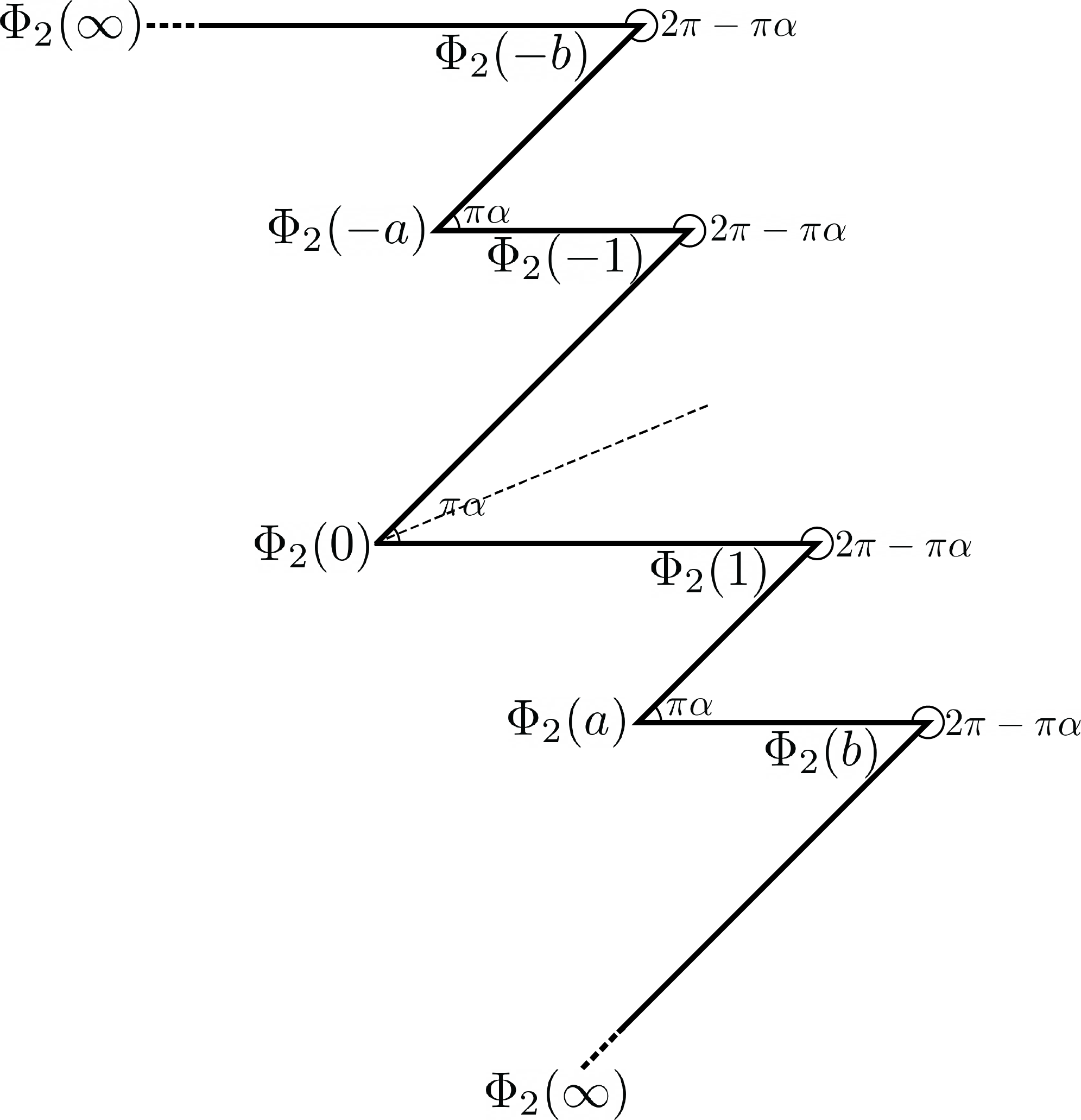}}
	\end{center}
	\caption{Flat Structures for $\alpha>0$}
	\label{fig:DEflat}

\end{figure}
The forms $\varphi_{1}$ and $\varphi_{2}$ are positive on $(0,1)$ which in turn determines our choice of rotations and branch cuts. Also note that the Schwarz-Christoffel maps with $\varphi_{1}$ and $\varphi_{2}$ for integrands, send the imaginary line to the symmetry axes of the polygons in Figure \ref{fig:DEflat}.
	
	Then we choose these 1-forms to be the Weierstrass 1-forms $Gdh=\rho\varphi_{1}$,  $\frac{1}{G}dh=\frac{1}{\rho}\varphi_{2}$ with a suitable L{\'o}pez-Ros factor $\rho$ which will be chosen later in our argument. As a result we get $dh=dz$, and clearly  $\Omega_3$ maps the imaginary and real lines to the imaginary and real lines, respectively.  Then by Proposition \ref{prop:symmetries}, $f$ will map the upper half plane to a minimal octagon as described in the statement of Lemma \ref{DEoctagon}.
	
	We note that, our construction of the Schwarz-Christoffel maps allows us to choose $\alpha=0$, for which we obtain the following 1-forms:
	\begin{equation*}
		Gdh=\rho
		\frac{z(a^2-z^2)}{(1-z^2)(b^2-z^2)}dz ,
		\qquad	
		\frac{1}{G}dh=
		\frac{(b^2-z^2)(1-z^2)}{\rho z(a^2-z^2)}dz.
	\end{equation*}
	Due to the assumption that $0<1<a<b$, we calculate the residues of $Gdh$, $\frac{1}{G}dh$ at the points $\{-b,-a,-1,0,1,a,b\}$ and see that $\phi_1$ maps the upper half plane to the polygon in Figure \ref{fig:DEflatlimit}(a) and $\phi_2$ maps the upper half plane to the polygon in Figure \ref{fig:DEflatlimit}(b).
	\begin{figure}[H]
		\begin{center}
			\subfigure[$Gdh$]{\includegraphics[width=2.5in]{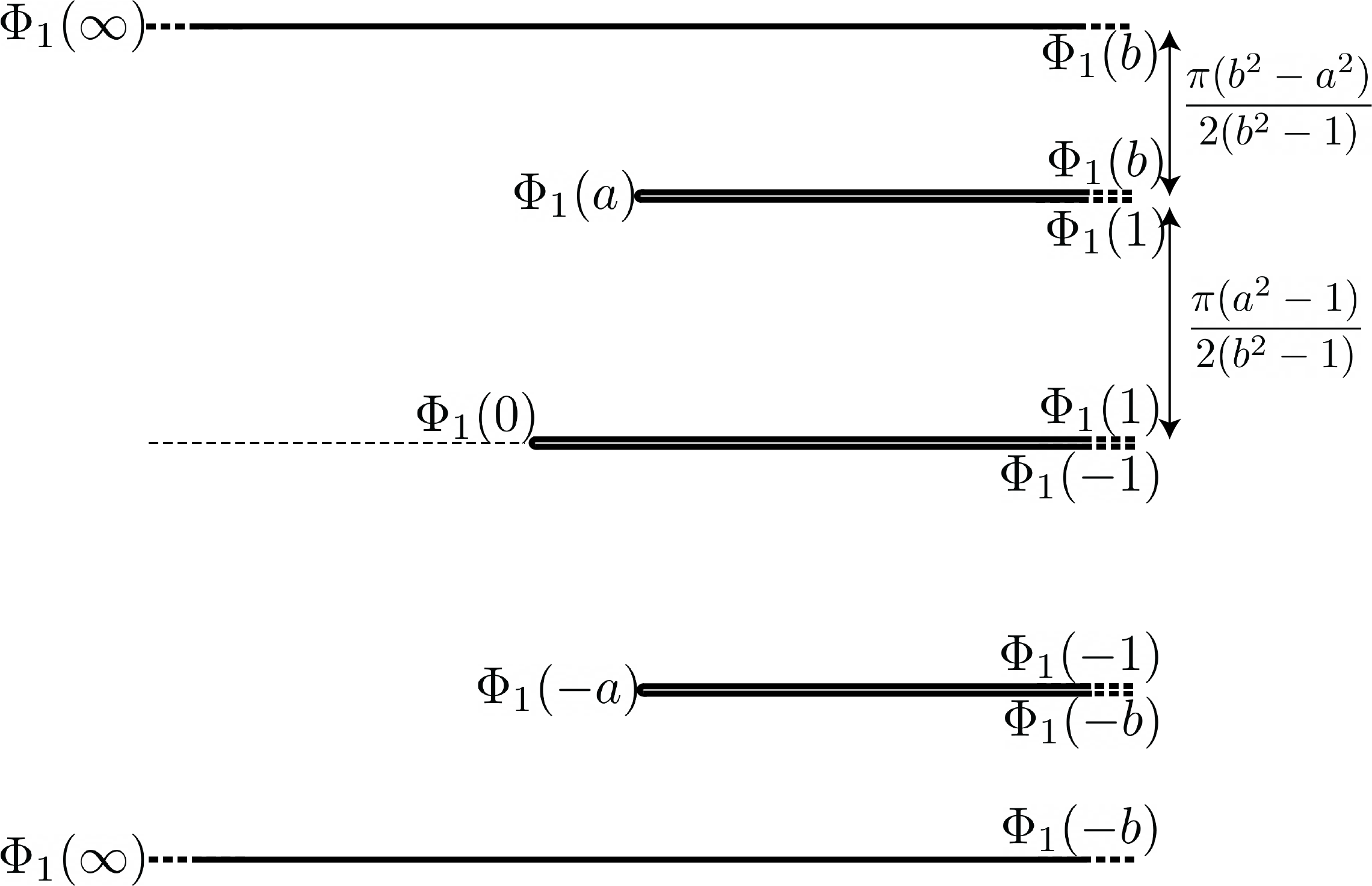}}
			\qquad\qquad
			\subfigure[$\frac{1}{G}dh$]{\includegraphics[width=2.5in]{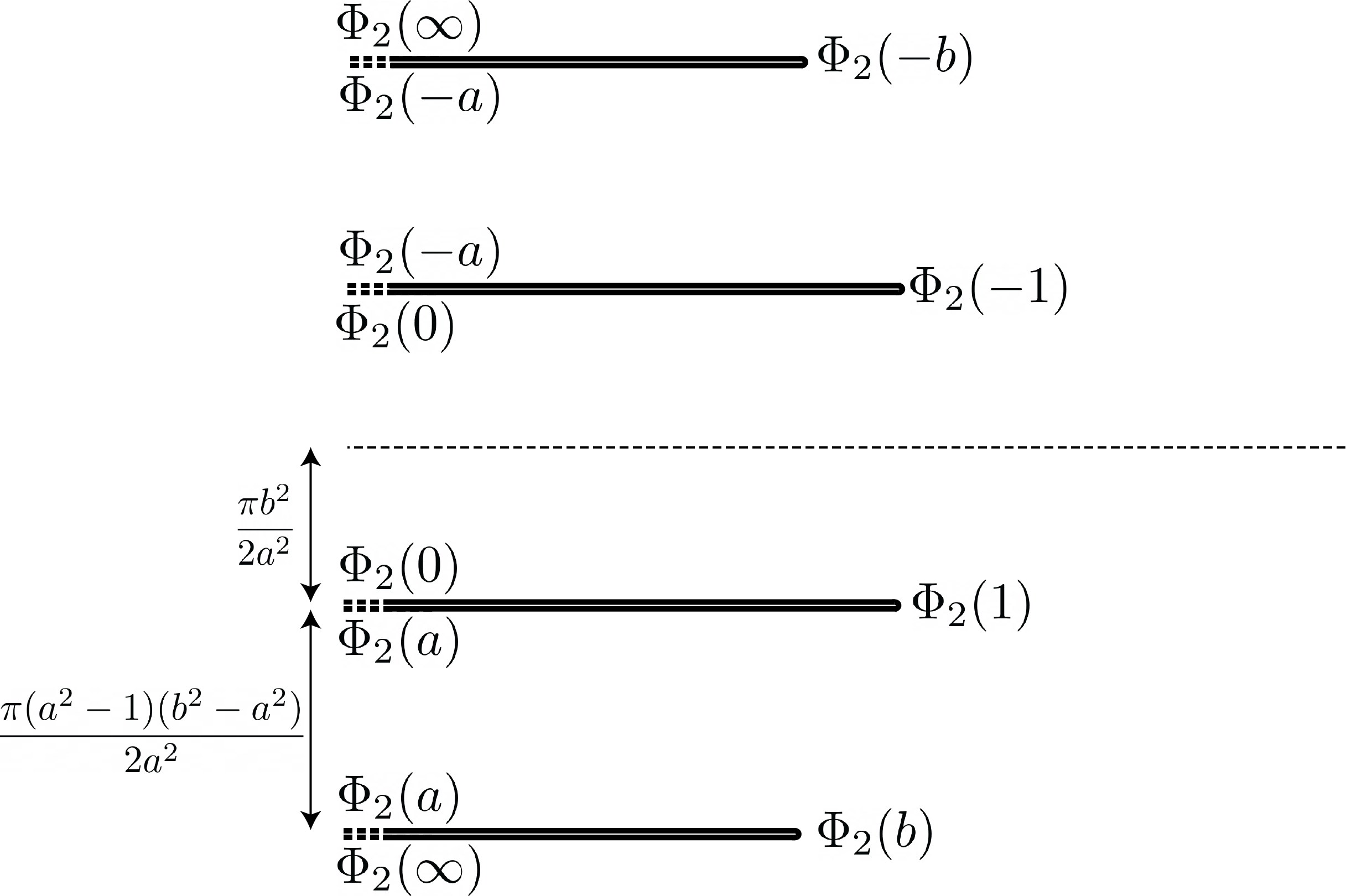}}
		\end{center}
	\caption{Flat structures when $\alpha=0$}
	\label{fig:DEflatlimit}

	\end{figure}
	Therefore, by Proposition \ref{prop:symmetries}, when $\alpha=0$, $f$ maps the upper half plane to a minimal octagon with all vertices at infinity and edges on parallel planes. \end{proof}
\subsection{Period Problem}\label{sec:DEperiod}
In order to extend the minimal octagon given by Lemma \ref{DEoctagon} to a complete wedge $W_{2\pi\alpha}$ corresponding to the $DE_{3,n}$ surface, it is sufficient and necessary that the alternating edges of the minimal octagon lie on the same vertical symmetry plane. We know that two alternating edges will lie on parallel planes and we can measure the distance between the parallel planes containing the two alternating edges as follows: Let $i<j$ be two consecutive numbers in $\{-b,-a,-1,0,1,a,b\}$, let $i_{p}$ be the immediate predecessor of  $i$ and $j_{s}$ be the immediate successor of $j$ in $\{-\infty,-b,-a,-1,0,1,a,b,\infty\}$. Also, let $\gamma_{i,j}$ be a circle in $\C$ that starts at a point in the interval $(j,j_{s})$, moves in the counterclockwise direction and passes through a point in the interval $(i_{p},i)$. That is, $\gamma_{i,j}$ encircles $i,j$ and $|f(i_{p})f(i)|$, $|f(j)f(j_{s})|$ are alternating edges of the minimal octagon. As in Proposition \ref{prop:symmetries}, we can extend i.e. analytically continue $G$ and $dh$ across $(i_{p},i)$. This corresponds to extending the minimal octagon by a vertical reflection with respect to the plane that contains the edge $|f(i_{p})f(i)|$. Then $\re {\int_{\gamma_{i,j}}(\omega_{1},\omega_{2})}$ is equal to the twice of the normal vector between the plane containing $|f(i_{p})f(i)|$ and the plane containing $|f(j)f(j_{s})|$. Since both $\varphi_{1}$ and $\varphi_{2}$ are symmetric with respect to the imaginary line, it is clear that if this normal vector is $0$ for $(i,j)=(1,a)$ and $(i,j=a,b)$, then it is $0$ simultaneously for$(i,j)=(-a,-1)$ and $(i,j)=(-b,-a)$. Hence for a given $\alpha \geq 0$, to be able to extend the minimal octagon by closing the period $\re {\int_{\gamma_{i,j}}(\omega_{1},\omega_{2})}$, we need values of $a,b,\rho$ such that 
\begin{equation}\label{DEperiod}
	\int_{\gamma_{0,1}}Gdh=\conj{\int_{\gamma_{0,1}}\frac{1}{G}dh}, \qquad \int_{\gamma_{1,a}}Gdh=\conj{\int_{\gamma_{1,a}}\frac{1}{G}dh}, \qquad\int_{\gamma_{a,b}}Gdh=\conj{\int_{\gamma_{a,b}}\frac{1}{G}dh}.
\end{equation}

So far, we have parametrized the surface $W_{2\pi \alpha}$ with 3 parameters $\rho,a,b$ and we have 3 period equations to solve. We can solve one these conditions, choose the L{\'o}pez-Ros factor $\rho$ and reduce the number of parameters essentially by scaling as follows:
\begin{lemma}\label{lem:DEperiods}
	Assume that for a given $\alpha \geq 0$, there exist $1<a<b$ such that the following holds true
\begin{equation}\label{eqn:DEperiod}
	\frac{\int_{\gamma_{0,1}}\varphi_{1}}{\int_{\gamma_{1,a}}\varphi_{1}}=\frac{\conj{\int_{\gamma_{0,1}}\varphi_{2}}}{\conj{\int_{\gamma_{1,a}}\varphi_{2}}}
	 \quad and \quad 	\frac{\int_{\gamma_{1,a}}\varphi_{1}}{\int_{\gamma_{a,b}}\varphi_{1}}=\frac{\conj{\int_{\gamma_{1,a}}\varphi_{2}}}{\conj{\int_{\gamma_{a,b}}\varphi_{2}}}
\end{equation}
where $\varphi_{1}$ and $\varphi_{2}$ are continued analytically along the circles $\gamma_{i,i+1}$. Then there exists a $\rho>0$ such that for $Gdh=\rho\varphi_{1}$, $\frac{1}{G}dh=\frac{1}{\rho}\varphi_{2}$, the period conditions given in \eqref{DEperiod} are satisfied.
\end{lemma}
\begin{proof}[Proof of Lemma \ref{lem:DEperiods}]
Note that $-e^{-i\alpha\pi}\int_{\gamma_{0,1}}\varphi_{1}$ and $-e^{-i\alpha\pi}\conj{\int_{\gamma_{0,1}}\varphi_{2}}$ are positive. \\
We define $\rho=\sqrt{\conj{\int_{\gamma_{0,1}}\varphi_{2}}\Big/ \int_{\gamma_{0,1}}\varphi_{1}}$. Observe that,
\begin{align*}
	\int_{\gamma_{0,1}}Gdh=\int_{\gamma_{0,1}}\rho\varphi_{1}=-e^{i\alpha\pi}\sqrt{\Big(\conj{\int_{\gamma_{0,1}}e^{-i\alpha\pi}\varphi_{2}}\Big)\Big( \int_{\gamma_{0,1}}e^{-i\alpha\pi}\varphi_{1}\Big)}&=\conj{\int_{\gamma_{0,1}}\frac{1}{\rho}\varphi_{2}}
	=\conj{\int_{\gamma_{0,1}}\frac{1}{G}dh}.
\end{align*}
Similarly,
\begin{equation*}
	\int_{\gamma_{1,a}}Gdh=\int_{\gamma_{1,a}}\rho\varphi_{1}=\sqrt{\frac{\conj{\int_{\gamma_{0,1}}\varphi_{2}}}{\int_{\gamma_{0,1}}\varphi_{1}}} \frac{\int_{\gamma_{0,1}}\varphi_{1}\conj{\int_{\gamma_{1}}\varphi_{2}}}{\conj{\int_{\gamma_{0,1}}\varphi_{2}}}=\conj{\int_{\gamma_{1,a}}\frac{1}{\rho}\varphi_{2}}=\conj{\int_{\gamma_{1,a}}\frac{1}{G}dh},
\end{equation*}
and finally,
\begin{align*}
	\int_{\gamma_{a,b}}Gdh=\int_{\gamma_{a,b}}\rho\varphi_{1}&=\sqrt{\frac{\conj{\int_{\gamma_{0,1}}\varphi_{2}}}{\int_{\gamma_{0,1}}\varphi_{1}}} \frac{\int_{\gamma_{1,a}}\varphi_{1}\conj{\int_{\gamma_{a,b}}\varphi_{2}}}{\conj{\int_{\gamma_{1,a}}\varphi_{2}}}\\
	&=\sqrt{\frac{\conj{\int_{\gamma_{0,1}}\varphi_{2}}}{\int_{\gamma_{0,1}}\varphi_{1}}} \frac{\int_{\gamma_{0,1}}\varphi_{1}\conj{\int_{\gamma_{a,b}}\varphi_{2}}}{\conj{\int_{\gamma_{0,1}}\varphi_{2}}}\\
	&=\conj{\int_{\gamma_{a,b}}\frac{1}{\rho}\varphi_{2}}=\conj{\int_{\gamma_{a,b}}\frac{1}{G}dh}.
\end{align*}
\end{proof}
Next, we will address the period problem and complete the proof of Theorem \ref{DEtheorem}.

\begin{lemma}\label{lem:DEwedge}
	For small enough $\alpha\geq0$, there exists a minimal wedge $W_{2\pi\alpha}$ corresponding to the $DE_{3,n}$ surface given in Theorem \ref{DEtheorem}.
\end{lemma}
\begin{proof}
	We define the map
	\begin{equation}
		P(a,b,\alpha):=\Bigg\{ 		\frac{\int_{\gamma_{0,1}}\varphi_{1}}{\int_{\gamma_{1,a}}\varphi_{1}}-\frac{\conj{\int_{\gamma_{0,1}}\varphi_{2}}}{\conj{\int_{\gamma_{1,a}}\varphi_{2}}}
		,			\frac{\int_{\gamma_{1,a}}\varphi_{1}}{\int_{\gamma_{a,b}}\varphi_{1}}-\frac{\conj{\int_{\gamma_{1,a}}\varphi_{2}}}{\conj{\int_{\gamma_{a,b}}\varphi_{2}}}
		\Bigg\}.
	\end{equation}	

Note that $P(a,b,0)$ can be expressed, using the residue theorem, in the following way:
	\begin{align*}
		&=\Bigg\{
		\frac{\Res(\varphi_{1},0)+\Res(\varphi_{1},1)}{\Res(\varphi_{1},1)+\Res(\varphi_{1},a)}-
		\frac{\conj{\Res(\varphi_{2},0)+\Res(\varphi_{2},1)}}{\conj{Res(\varphi_{2},1)+\Res(\varphi_{2},a)}},\\
		&\hspace{1in} \frac{\Res(\varphi_{1},1)+\Res(\varphi_{1},a)}{\Res(\varphi_{1},a)+\Res(\varphi_{1},b)}-
		\frac{\conj{\Res(\varphi_{2},1)+\Res(\varphi_{2},a)}}{\conj{Res(\varphi_{2},a)+\Res(\varphi_{2},b)}}
		\Bigg\}\\
		&=\Big\{
		1+\frac{2b^2}{(a^2-1)(a^2-b^2)},-1+\frac{a^2-1}{b^2-a^2}
		\Big\}.
	\end{align*}
	Moreover, we have that $P(\sqrt{3+\sqrt{6}},\sqrt{5+2\sqrt{6}},0)=0$. In other words, we can solve the period problem when $\alpha=0$. We can now apply the Implicit Function Theorem to the differentiable function $P$ at $\alpha=0$. We exclude the computations that show that the Jacobian of $P$ at $\alpha=0$ is equal to $\frac{8ab(b^2+1)}{(a^2-1)(b^2-a^2)^3}$, since it can be found using standard arguments. Since it is clearly non-zero, we conclude that for every $\alpha$ in a neighborhood of $0$, we can indeed the solve the period conditions given in \eqref{DEperiod} for some $1<a<b$.
\end{proof}
\subsection{Extension and Types of Ends}\label{sec:type of ends}
 In this section we discuss the extension of $W_{2\pi\alpha}$ to $DE_{3,n}$ and the properties of the end point of $DE_{3,n}$. 
 
Observe that for an integer $n\geq2$ and $\alpha=\frac{1}{n}$, we can extend the minimal surface $W_{2\pi\alpha}$ by reflections to $DE_{3,n}$. Note that the $\frac{1}{2n}$th of the developed image of the extended $Gdh$ and $\frac{1}{G}dh$ in $\C$ are the octagons in Figure \ref{fig:DEflat}. We observe that $\phi_1(\gamma_{\infty})$ and $\phi_2(\gamma_{\infty})$ have winding numbers $-\frac{\alpha}{2}$ and $-(1-\frac{\alpha}{2})$ respectively, where $\gamma_{\infty}$ is a half circle around $\infty$ in the the upper half plane. Thus for the extended Weierstrass data, images of a circle  under the maps $\int Gdh$ and$\int \frac{1}{G}dh$ have winding numbers $-1$ and $(1-\frac{2}{\alpha})$. Hence $Gdh$ and $\frac{1}{G}dh$ have poles of order $2$ and $2n$, respectively. Consequently, $G$ has a zero of order $n-1$ and $dh$ has pole of order $n+1$ around the end point of the extended surface.  This implies that the extended surface has an enneper type end. 

On the other hand, when $\alpha=0$, after reflecting the minimal octagon once, we obtain a singly periodic surface which is defined on $\C$ by
\begin{equation*}
	G=
	\rho\frac{z(a^2-z^2)}{(1-z^2)(b^2-z^2)}dz, 
	\qquad	
	dh=dz.
\end{equation*}
We see that $G$ has simple zeros at $-a,0,a,\infty$ and simple poles at $-b,-1,1,b$ whereas $dh$ only has a second order pole at $\infty$. This implies that the dihedral limit $W_0$ of $DE_{3,n}$ has Scherk type ends at the points $-b,-a,-1,0,1,a,b$ and an enneper type end at $\infty$.

Now, for any sufficiently large $n\in\N$,  we use Lemma \ref{lem:DEwedge} to solve the period conditions for the minimal wedge $W_{2\pi\alpha}$ with $\alpha=\frac{1}{n}$. Then the extended surface $DE_{3,n}$ is the Chen-Gackstatter Surface of genus $3(n-1)$ with dihedral symmetry as claimed in Theorem \ref{DEtheorem}. With this we conclude the proof of Theorem \ref{DEtheorem}.

\end{proof}


\section{Dihedralized Weber Wolf Surface}\label{sec:DWW}
In this section we will give an alternative existence proof of Weber-Wolf Surface with catenoidal ends, $DH_{1,1}$ with $n$-fold dihedral symmetry and with 4 catenoidal ends. These surfaces, along with many of their variants, were first introduced in \cite{ww2} by using Teichmuller theoretical methods to construct. Our proof provides the existence of a specific case mentioned above in a brief manner.
\begin{theorem}\label{DWWtheorem}
	For sufficiently large values of $n\in \N$, there exists a finite type $n$-dihedral minimal surface $DH_{1,1,n}$ of genus $3n-3$ with four catenoidal ends and with a planar end which is invariant under $180^\circ$ horizontal rotations. Moreover, as $n \rightarrow \infty$, $DH_{1,1,n}$ converges to a singly periodic minimal surface of genus 0 with 8 annular ends, the dihedral limit surface is invariant under vertical translations and $180^\circ$ rotations around two horizontal lines.       
\end{theorem}
\begin{figure}[H]
	\begin{center}
		\includegraphics[width=6cm]{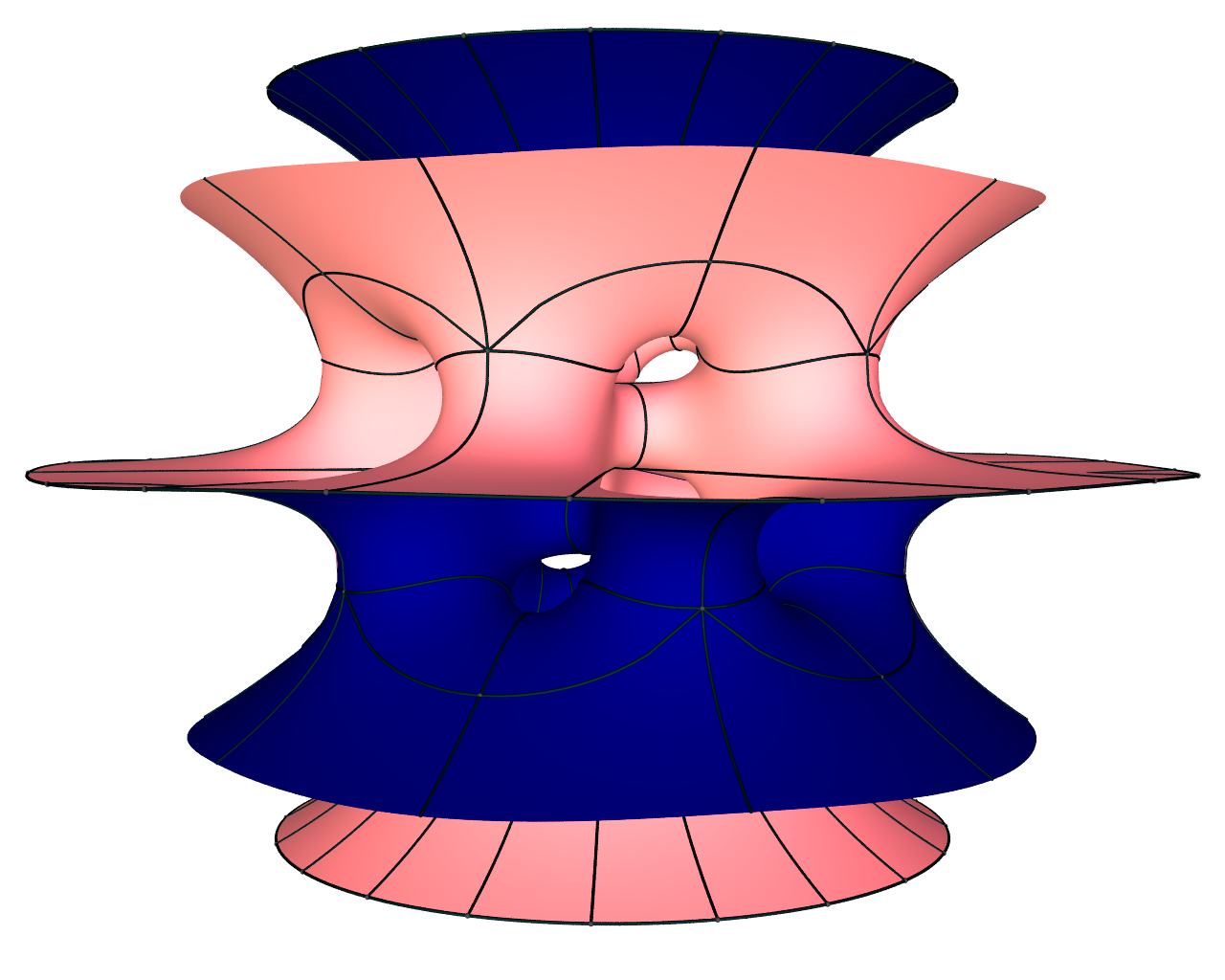}
	\end{center}
	\caption{Dihedralized Weber-Wolf Surface, $\alpha=\sfrac{1}{5}$}
	\label{fig:DWW}	
\end{figure}
\begin{proof}
	As in the case with $DE_{3,n}$, the quotient of $DH_{1,1,n}$ under rotational symmetry group is conformally equal to $\hat{\C}$. Hence the proof structure is identical to proof Theorem \ref{DEtheorem}. First, we construct a half of the wedge $W_{2\pi\alpha}$ corresponding to  $DH_{1,1,n}$ in the following lemma. 
\begin{lemma}\label{lem:DWWoctagon}
	For any $0<\rho$, $0\leq\alpha$ and $0<1<a<b<c<\infty$, there is a Weierstrass map $f$, cf. Section \ref{sec:geom}, that maps the upper half plane to a minimal octagon with consecutive 5 vertices at $\infty$ in $\R^3$. The edges of this octagon lie in vertical symmetry planes either parallel to the $x$-axis or make an angle $\alpha\pi$ with the $x$-axis alternatingly (if $\alpha=0$, then all the symmetry planes are parallel to the $x$-axis and all the vertices of the minimal octagons are at $\infty$). Moreover, there is a half straight line on this octagon that passes from two of the vertices and divides the octogon into two symmetrical pieces.  
\end{lemma}
\begin{figure}[H]
	\begin{center}
		\includegraphics[width=6cm]{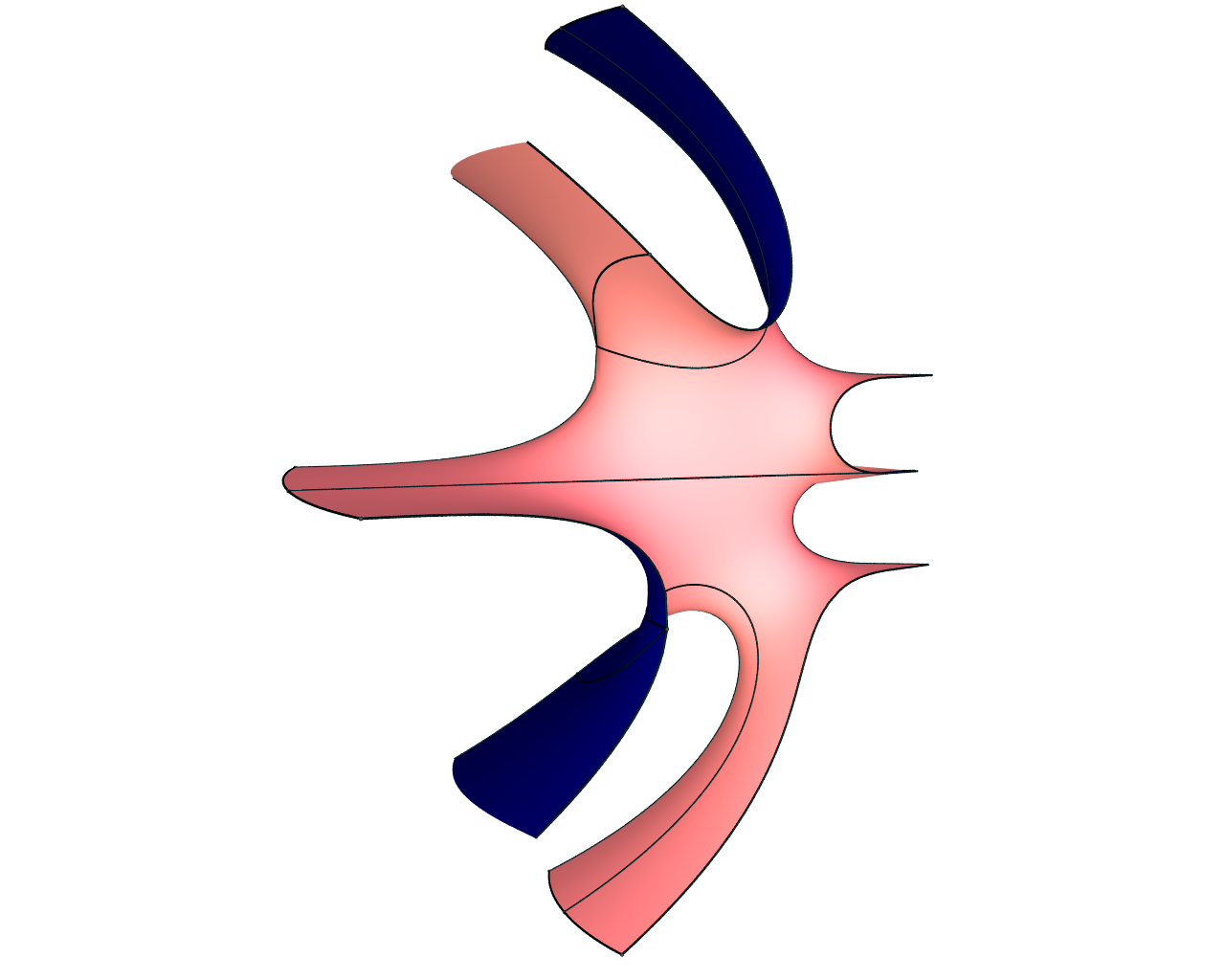}
	\end{center}
	\caption{A minimal octagon in $\R^3$ corresponding to $DH_{1,1,n}$}
	\label{fig:DWWoctagon}
\end{figure}2
	\begin{proof}
		As in the previous case, we begin the proof by constructing Schwarz-Christoffel maps to symmetric the polygons in Figure \ref{fig:DWWflat} where the positive imaginary line in $\C$ is mapped to the symmetry axes given by dashed lines in the figures. Define,
		\begin{equation}\begin{split}\label{DWWintegrands}
				\varphi_{1}=
				z^{\alpha-1}(1-z^2)^{1-\alpha}(a^2-z^2)^{\alpha-1}(c^2-z^2)^{-1-\alpha}dz, \\
				\varphi_{2}=
				z^{1-\alpha}(1-z^2)^{\alpha-1}(a^2-z^2)^{-1-\alpha}(b^2-z^2)^2(c^2-z^2)^{\alpha-1}dz.
		\end{split}\end{equation}
		The Schwarz-Christoffel map, with integrand $\varphi_1$ defined in \eqref{DWWintegrands}, sends the upper half plane to the right side of the polygon in Figure \ref{fig:DWWflat}(a). Similarly, the Schwarz-Christoffel map, with integrand $\varphi_2$ defined in \eqref{DWWintegrands}, sends the upper half plane to the left side of the polygon in Figure \ref{fig:DWWflat}(b). Thus, we choose the Weierstrass $1$-forms given by $Gdh=\rho\varphi_{1}$, $\frac{1}{G}dh=\frac{1}{\rho}\varphi_{2}$ and a suitable $\rho$ will be chosen in the following arguments.

		\begin{figure}[H]
			\begin{center}
				\subfigure[$Gdh$]{\includegraphics[width=2.5in]{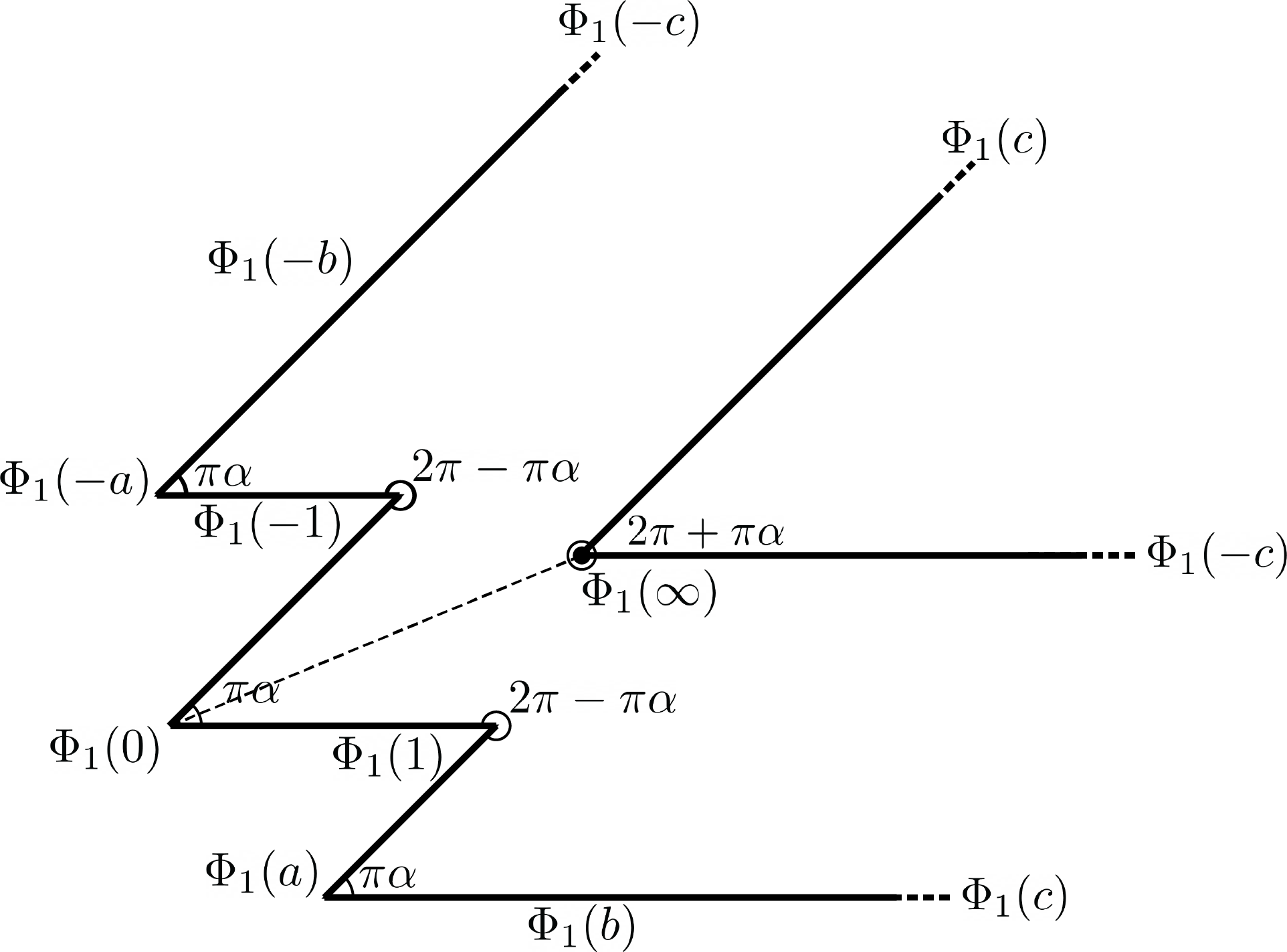}}
				\qquad\qquad\qquad
				\subfigure[$\frac{1}{G}dh$]{\includegraphics[width=2.5in]{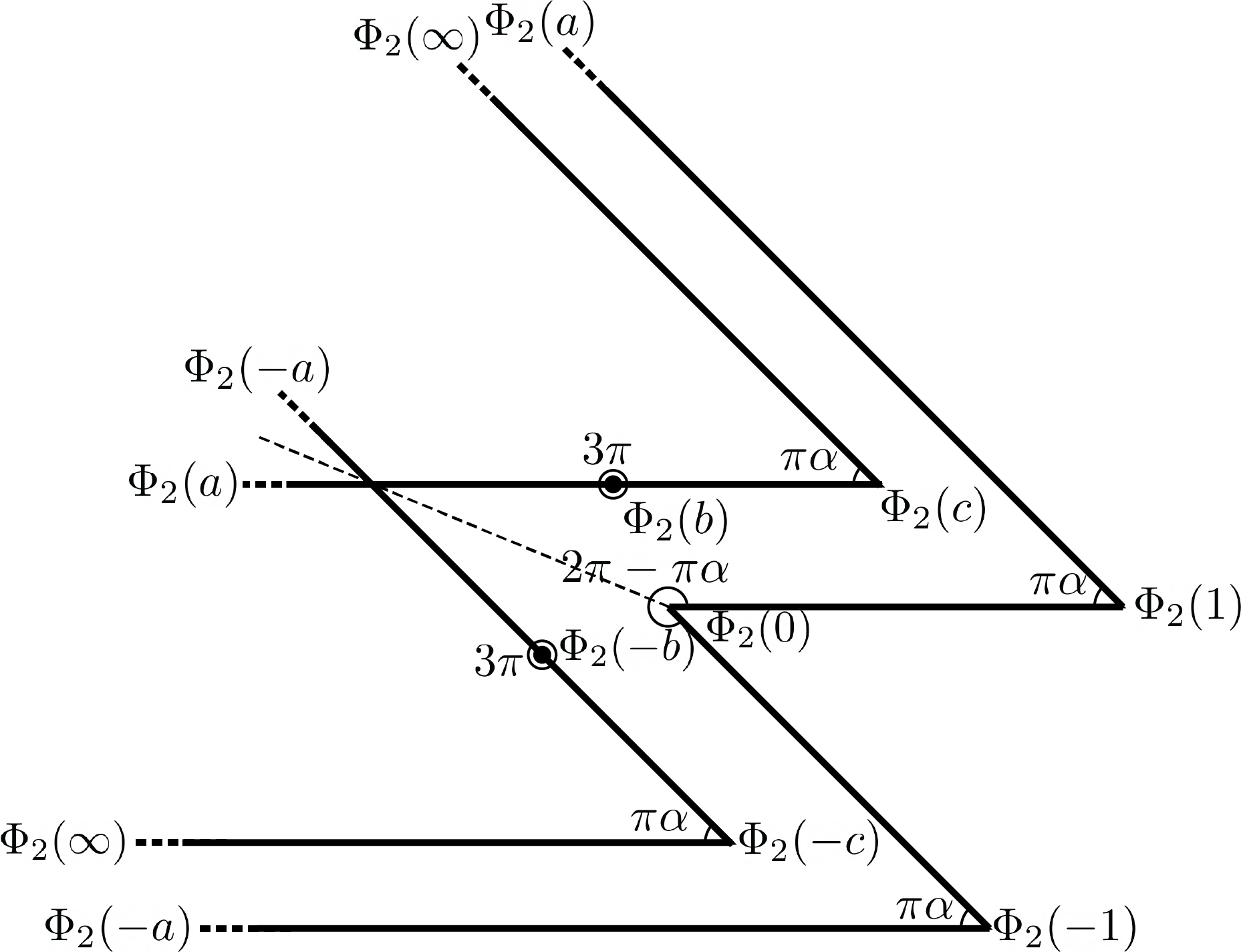}}
			\end{center}
			\caption{$DWW_n$ Flat Structures}
			\label{fig:DWWflat}

		\end{figure}

	Hence, the Gauss map and the height differential are given by
		\begin{align*}
			G&=\rho
			z^{\alpha-1}
			(1-z^2)^{1-\alpha}
			(a^2-z^2)^{\alpha}
			(b^2-z^2)^{-1}
			(c^2-z^2)^{-\alpha},\\
			dh&=\frac{(b^2-z^2)}{(a^2-z^2)(c^2+z^2)}dz.
		\end{align*}
		Observe that $dh$ assumes real values on the real line. As a consequence of Proposition \ref{prop:symmetries}(2), we see that the Weierstrass map $f$ maps the real line to vertical symmetry curves. Thus Proposition \ref{prop:symmetries}(3) implies that $f$ will map the upper half place to a minimal octagon as described in the statement of Lemma \ref{lem:DWWoctagon}.
		
		Then we finish the proof of Lemma \ref{lem:DWWoctagon} by using our symmetry Proposition \ref{prop:symmetries} again and by describing the limits of the Schwarz-Christoffel maps $\Phi_i$ and the limits of polygons given in Figure \ref{fig:DWWflat}, as $\alpha$ goes to $0$. The limits of the Schwarz-Christoffel maps are given by
		\begin{equation*}
			Gdh=
			\frac{1-z^2}{z(c^2-z^2)(a^2-z^2)}dz ,
			\qquad	
			\frac{1}{G}dh=
			\frac{z(b^2-z^2)^2}{(1-z^2)(c^2-z^2)(a^2-z^2)}dz.
		\end{equation*}
	Which map the upper half plane to the polygons
		\begin{figure}[H]
			\begin{center}
				\subfigure[$Gdh$]{\includegraphics[width=2.8in]{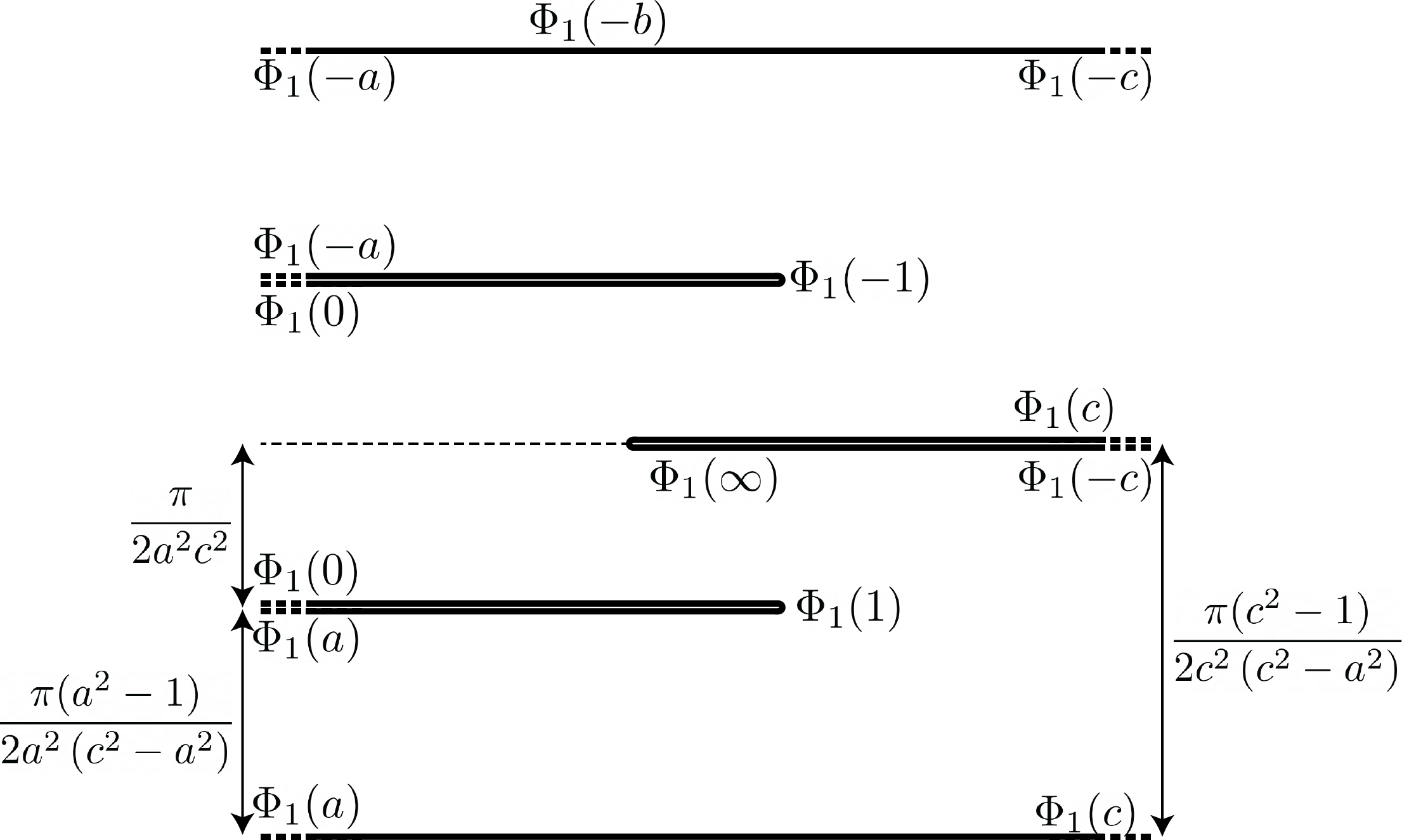}}
				\subfigure[$\frac{1}{G}dh$]{\includegraphics[width=3.2in]{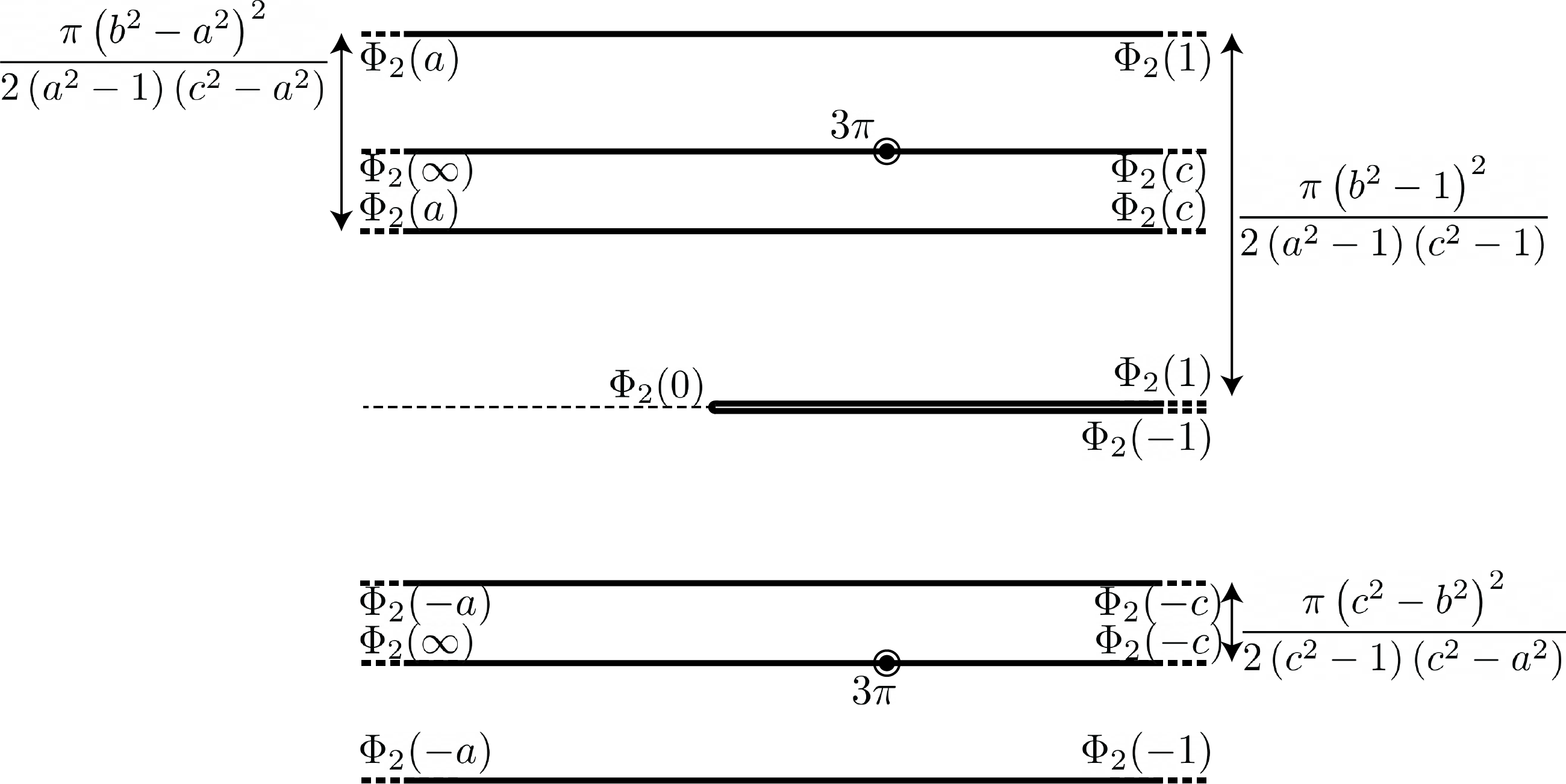}}
			\end{center}
			\caption{$DW_n$ flat structures as $\alpha \rightarrow 0$}\label{fig:DWWlimitflat}	
		\end{figure}
	\end{proof}

	\subsection{Period Problem}	\label{DCWperiodproblem}
	In order to extend the minimal octagon, obtained in Lemma \ref{lem:DCCWhexagon}, to a minimal wedge $W_{2\pi\alpha}$, we need to show that the sides of the octagon in Figure \ref{fig:DWWoctagon} that make the same angle with the $x$-axis, lie on the same plane. Similar to the extension of the half wedge corresponding to the $DE_{3,n}$ surface found in Section \ref{sec:DEperiod} and Lemma \ref{lem:DEperiods}, this can be achieved by satisfying the condition 
\begin{equation}\label{eqn:DWW
		period}
	\frac{\int_{\gamma_{0,1}}\varphi_{1}}{\int_{\gamma_{1,a}}\varphi_{1}}=\frac{\conj{\int_{\gamma_{0,1}}\varphi_{2}}}{\conj{\int_{\gamma_{1,a}}\varphi_{2}}}
	\quad and \quad 	\frac{\int_{\gamma_{1,a}}\varphi_{1}}{\int_{\gamma_{a,c}}\varphi_{1}}=\frac{\conj{\int_{\gamma_{1,a}}\varphi_{2}}}{\conj{\int_{\gamma_{a,c}}\varphi_{2}}}
\end{equation}	
	where $\gamma_{i,j}$ is as described in Section \ref{sec:DEperiod}. Note that for $\alpha=0$, the conditions in \eqref{DCCWperiod} are also valid for the limiting octagon. 

\begin{lemma}\label{lem:DWWwedge}
	For small enough $\alpha\geq0$, there exists a minimal wedge $W_{2\pi\alpha}$ corresponding to the $DH_{1,1,n}$ surface given in Theorem \ref{DWWtheorem}.
\end{lemma}
\begin{proof}
	We define the map
	\begin{equation}
		P(a,b,c,\alpha):=\Bigg\{ 		\frac{\int_{\gamma_{0,1}}\varphi_{1}}{\int_{\gamma_{1,a}}\varphi_{1}}-\frac{\conj{\int_{\gamma_{0,1}}\varphi_{2}}}{\conj{\int_{\gamma_{1,a}}\varphi_{2}}}
		,			\frac{\int_{\gamma_{1,a}}\varphi_{1}}{\int_{\gamma_{a,c}}\varphi_{1}}-\frac{\conj{\int_{\gamma_{1,a}}\varphi_{2}}}{\conj{\int_{\gamma_{a,c}}\varphi_{2}}}
		\Bigg\}.
	\end{equation}	
	
	The first coordinate of $P(a,b,0)$, by the residue theorem, is equal to
	\begin{equation*}
		\frac{\int_{\gamma_{0,1}}\varphi_{1}}{\int_{\gamma_{1,a}}\varphi_{1}}-\frac{\conj{\int_{\gamma_{0,1}}\varphi_{2}}}{\conj{\int_{\gamma_{1,a}}\varphi_{2}}}=\frac{\left(a^2-c^2\right) \left(2 a^2 \left(c^2-1\right)+b^4 \left(c^2+2\right)-6 b^2 c^2+3 c^2\right)}{\left(a^2-1\right) c^2 \left(a^2 \left(c^2-1\right)+b^4-2 b^2 c^2+c^2\right)},	
	\end{equation*}
	where is the second coordinate is equal to
	\begin{equation*}
		\frac{\int_{\gamma_{1,a}}\varphi_{1}}{\int_{\gamma_{a,c}}\varphi_{1}}-\frac{\conj{\int_{\gamma_{1,a}}\varphi_{2}}}{\conj{\int_{\gamma_{a,c}}\varphi_{2}}}=-\frac{\left(a^2-1\right) \left(a^2 \left(c^2-1\right)^2-\left(b^4 \left(c^2+1\right)\right)+4 b^2 c^2-c^2 \left(c^2+1\right)\right)}{\left(a^2-c^2\right) \left(a^2 \left(c^2-1\right)-b^4+2 b^2-c^2\right)}.
	\end{equation*}
In particular for $a=\frac{1}{5} \left(\left(2 \sqrt{6}+5\right)^{3/2}-8 \sqrt{2 \sqrt{6}+5}\right)$, $b=\sqrt{\frac{1}{5} \left(4 \sqrt{6}+9\right)}$ and $c=\frac{b^2}{a}$ we have $P(a,b,0)=0$ and Jacobian of P at $\alpha=0$ is nonzero. Therefore, we can extend the solution for $\alpha\geq0$ by Implicit Function Theorem where all the solution triples $(a,b,c)$ are of the form $(a,b,\frac{b^2}{a})$.
\end{proof}

\subsection{Extension, Types of Ends and Growth rates}	
Lemma \ref{lem:DWWwedge} shows that, for large enough n the period problem of $DH_{1,1,n}$ can be solved for $c=\frac{b^2}{a}$. By the winding number argument in Section \ref{sec:type of ends}, one can show that that the ends of $DH_{1,1,n}$ are catenoidal. In particular, the growth rate of catenoidal end around $a$ is $\frac{b^2-a^2}{2a(a^2-c^2)}$ and that around $c$ is $\frac{c^2-b^2}{2c(a^2-c^2)}$. Moreover, at $c=\frac{b^2}{a}$ the two growth rates are equal. That is to say that outside of a compact cylinder in $\R^3$, $DH_{1,1,n}$ looks like a union of two asymptotic concentric catenoids, and a horizontal plane that divides these catenoids into 2 symmetric pieces on their neck. On the other hand at $\alpha=0$, all the ends are replaced by Scherk type ends which are parallel to each other. Equal growth condition motivated the author to look for solutions in form $(a,b,\frac{b^2}{a})$. This concludes the proof of Theorem \ref{DWWtheorem}.
\end{proof}


\section{Dihedralized Catenoidal Costa-Wohlgemuth Surface}\label{sec:DCCW}

Having outlined the dihedralization argument in Sections \ref{sec:DE} and \ref{sec:DWW}, we shift our focus to proving the existence of two novel minimal surfaces, thereby supporting the merit of our method. While the proof of existence for the first surface closely aligns with previous examples, the latter necessitates additional arguments, given that we cannot employ the Residue Theorem in the limit case, as done in previous instances. We now discuss the proof of Theorem \ref{DCCWtheorem}.
\begin{proof}[Proof of Theorem \ref{DCCWtheorem}]
Similar to the situation with $DE_{3,n}$ and $DH_{1,1,n}$, the conformal equivalence of the quotient of $DCCW_n$ under rotational symmetry group is to $\hat{\C}$. Hence will construct $DCCW_n$ in the same way on the upper half plane by using the Schwarz-Christoffel maps. First, we construct a half of the wedge $W_{2\pi\alpha}$ corresponding to  $DCCW_n$ in the following lemma. 

\begin{lemma}\label{lem:DCCWhexagon}
	For any $\rho>0$, $\alpha\geq 0$ and $0<1<a<b<c<\infty$, there is a Weierstrass map $f$ that maps the upper half plane to a minimal hexagon with 4 vertices at $\infty$ in $\R^3$. The edges of this hexagon lie in vertical symmetry planes that are either parallel to or make an angle $\alpha\pi$ with the $x$-axis alternatingly (if $\alpha=0$, then all the symmetry planes are parallel to the $x$-axis and all the vertices of the minimal hexagons are at $\infty$ ). Moreover, this polygon is invariant under reflection with respect to a horizontal plane.  
\end{lemma}
\begin{figure}[H]
	\begin{center}
		\subfigure[$\alpha>0$]{\includegraphics[width=3in]{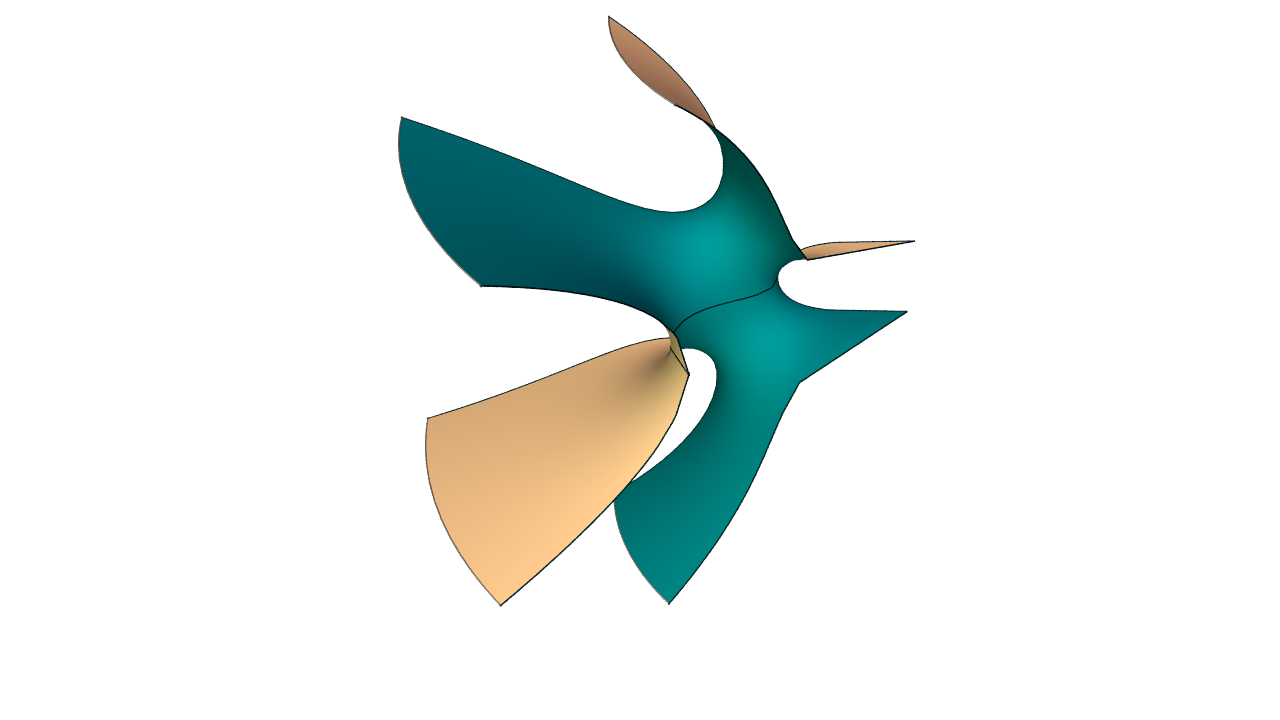}}
		\subfigure[$\alpha=0$]{\includegraphics[width=3in]{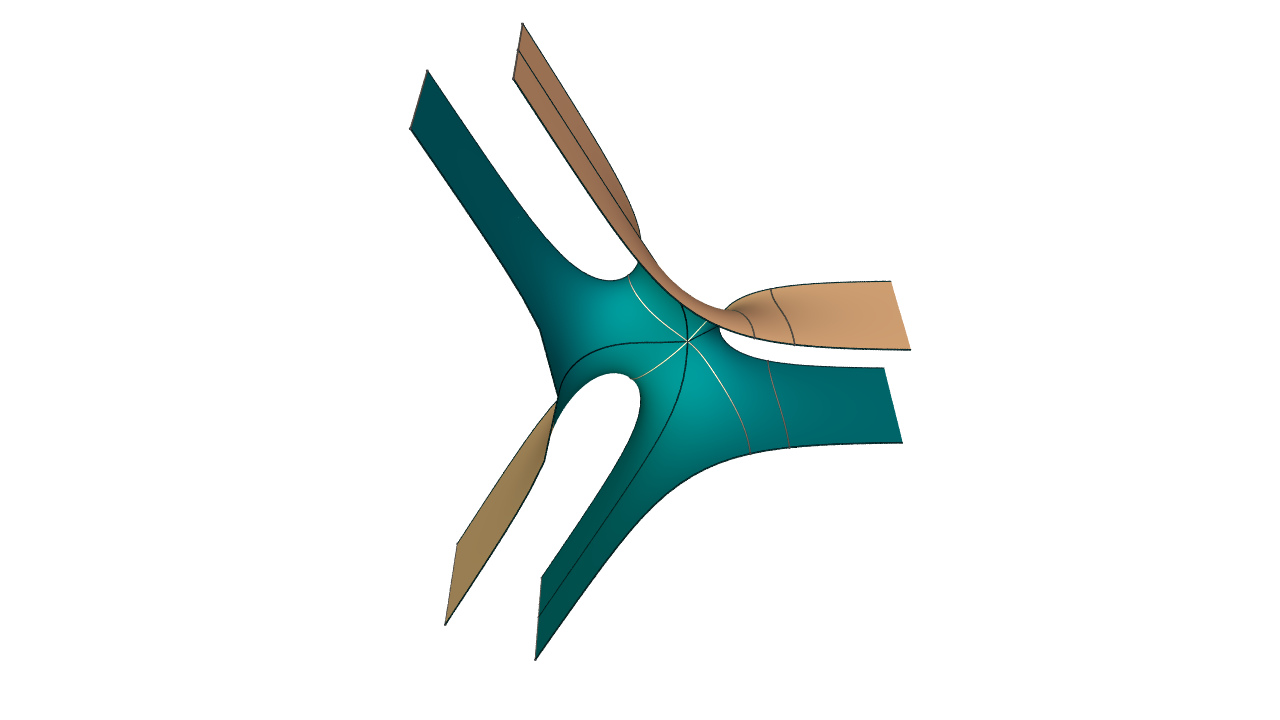}}
	\end{center}
	\caption{Minimal Hexagons in $\R^3$}
	\label{fig:DCCWhexagon}
\end{figure}
\begin{proof}
	We begin the proof by constructing Schwarz-Christoffel maps to the polygons in Figure \ref{fig:DCCWflat}. Define,
	\begin{equation}\begin{split}\label{DCCWintegrands}
			\varphi_{1}=
			(z+c)^{\alpha-1}
			(z+b)^{2}
			(z+a)^{-\alpha-1}
			(z+1)^{\alpha-1}
			(1-z)^{1-\alpha}
			(a-z)^{\alpha-1}
			(c-z)^{-\alpha-1}dz, \\
			\varphi_{2}=
			(z+c)^{-\alpha-1}
			(z+a)^{\alpha-1}
			(z+1)^{1-\alpha}
			(1-z)^{\alpha-1}
			(a-z)^{-\alpha-1}
			(b-z)^{2}
			(c-z)^{\alpha-1}dz .
	\end{split}\end{equation}
We can see that the Schwarz-Christoffel map, with integrand $\varphi_1$ defined in \eqref{DCCWintegrands}, sends the upper half plane to the right side of the polygon in Figure \ref{fig:DCCWflat}(a). Similarly, the Schwarz-Christoffel map, with integrand $\varphi_2$ defined in \eqref{DCCWintegrands}, sends the upper half plane to the left side of the polygon in Figure \ref{fig:DCCWflat}(b).
\begin{figure}[H]
	\begin{center}
		\subfigure[$Gdh$]{\includegraphics[width=2.5in]{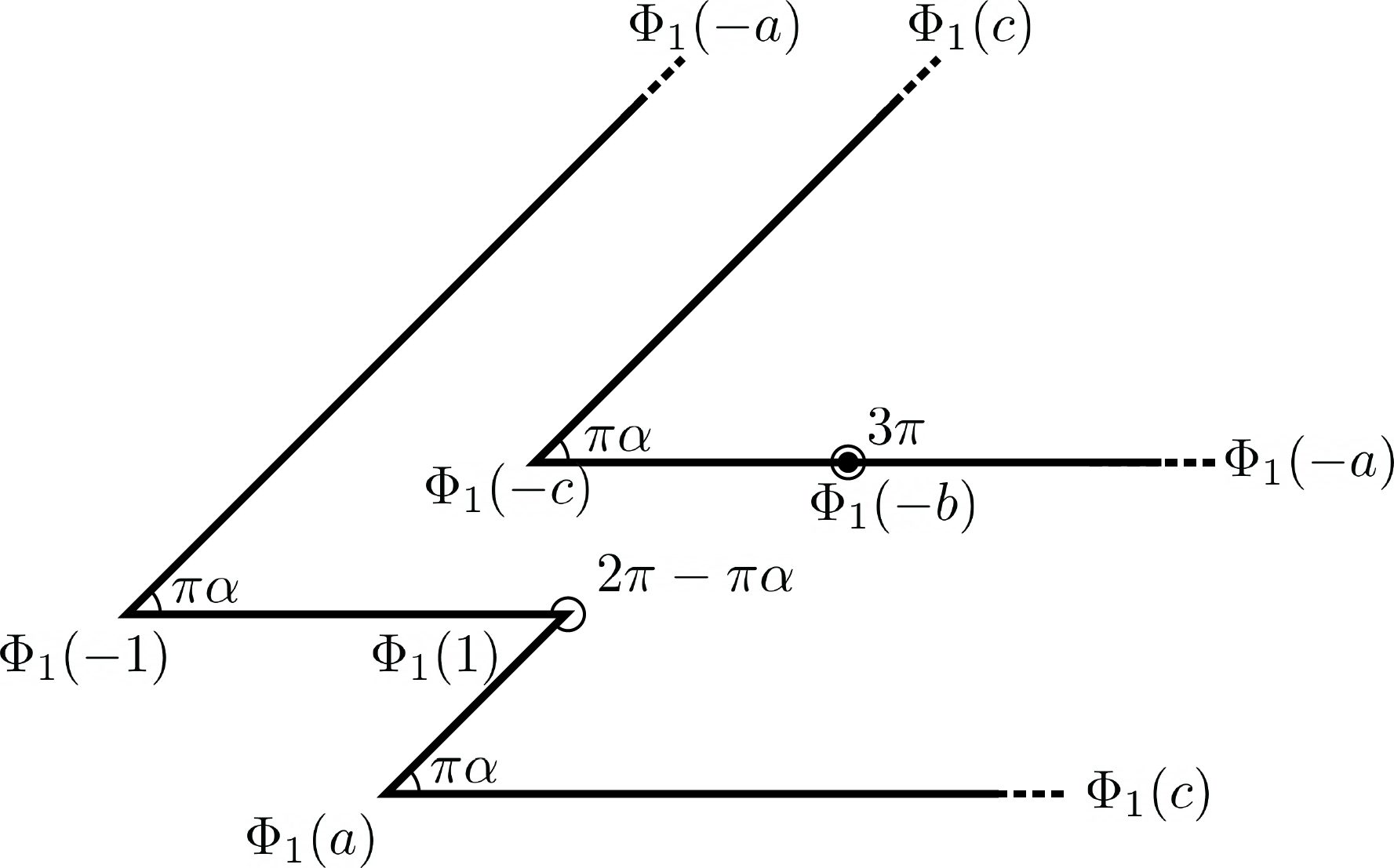}}
		\qquad\qquad\qquad
		\subfigure[$\frac{1}{G}dh$]{\includegraphics[width=2.5in]{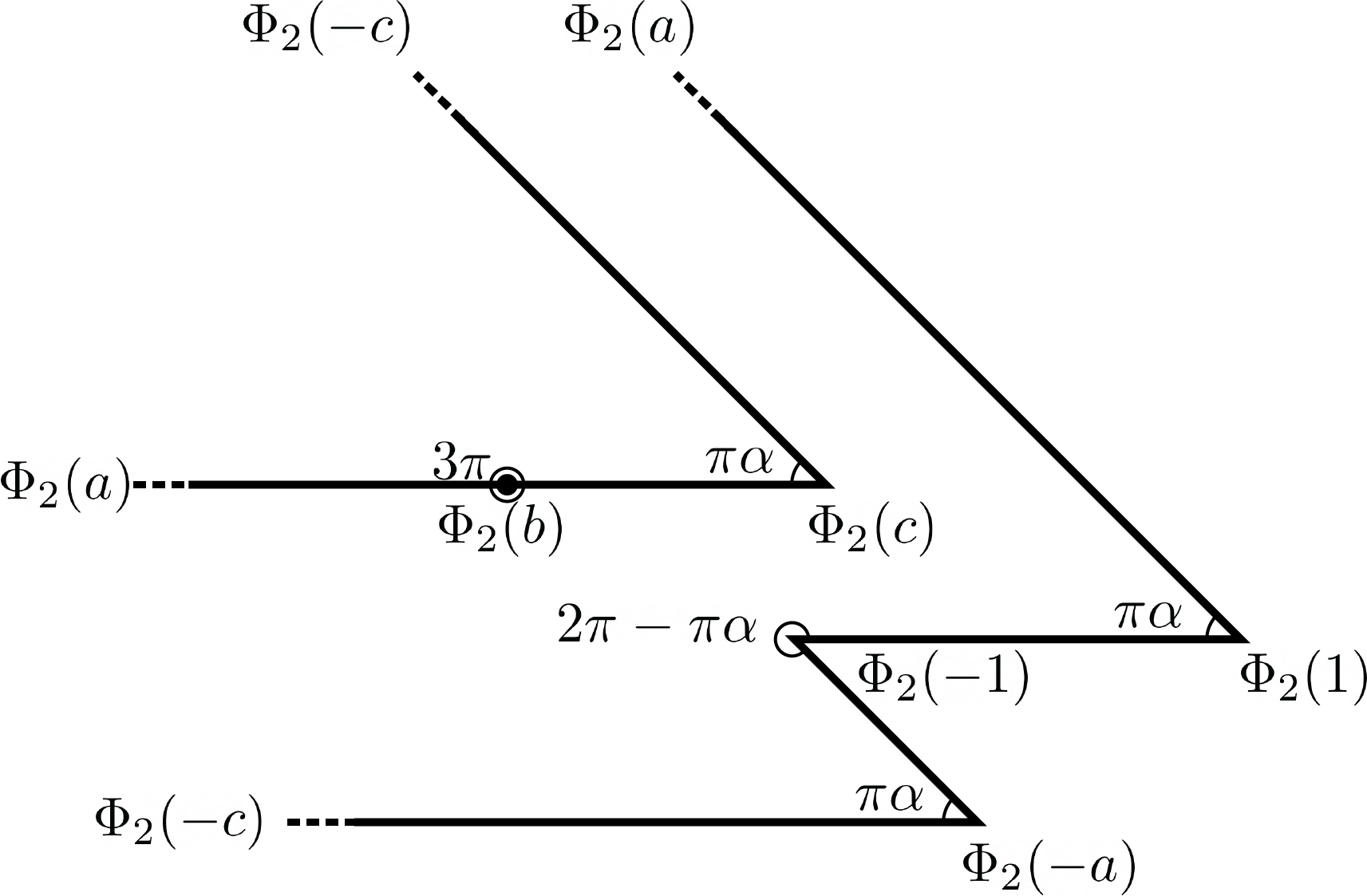}}
	\end{center}
	\caption{$DCCW_n$ Flat Structures}
	\label{fig:DCCWflat}

\end{figure}
We choose branch cuts that allow $\varphi_{i}$ to be positive on the interval $[-1,1]$. The forms $\varphi_i$ then become the Weierstrass $1$-forms $Gdh$ and $\frac{1}{G}dh$ after scaling by a real factor. Consequently, the Gauss map and the height differential are given by
\begin{align*}
	G=\rho(z+c)^{\alpha}
	(z+b)
	&(z+a)^{-\alpha}
	(z+1)^{\alpha-1}
	(1-z)^{1-\alpha}
	(a-z)^{\alpha}
	(b-z)^{-1}
	(c-z)^{-\alpha},\\
	dh&=\frac{(b-z)(z+b)}{(a-z)(a+z)(c+z)(c-z)}dz.
\end{align*}
Observe that $dh$ assumes real values on the real line. As a consequence of Proposition \ref{prop:symmetries}(2), we see that the Weierstrass map $f$ maps the real line to vertical symmetry curves. Also observe that for $\rho=1$, we have $\sigma^*G=\conj{\frac{1}{G}}$ and $\sigma^*dh=-\conj{dh}$, where $\sigma$ denotes the reflection on the imaginary line. In conjunction with Proposition \ref{prop:symmetries}(3), this implies that $\sigma$ can be realized by a reflection with respect to a horizontal plane in $\R^3$ and
the imaginary line in the upper half plane is mapped into this symmetry plane in $\R^3$. Hence, we deduce that the Weierstrass map $f$ maps the upper half plane to a minimal hexagon (cf. Figure \ref{fig:DCCWhexagon}).

As in Lemma \ref{DEoctagon}, we consider the limits of the Schwarz-Christoffel maps $\Phi_i$ and the limits of polygons given in Figure \ref{fig:DCCWflat}, as $\alpha$ goes to $0$. At $\alpha=0$, we have
\begin{equation*}
	Gdh=
	\frac{(z+b)^2(1-z)}{(z+1)(c^2-z^2)(a^2-z^2)}dz ,
	\qquad	
	\frac{1}{G}dh=
	\frac{(b-z)^2(z+1)}{(1-z)(c^2-z^2)(a^2-z^2)}dz.
\end{equation*}
Since $\Phi_{i}$ are well defined on $\C$, we can calculate the residues and show that, as $\alpha \rightarrow 0$, the polygons in Figure \ref{fig:DCCWflat}(a),(b) converge to the polygons in Figure \ref{fig:DCCWlimitflat}(a),(b) respectively.
\begin{figure}[H]
	\begin{center}
		\subfigure[$Gdh$]{\includegraphics[width=3in]{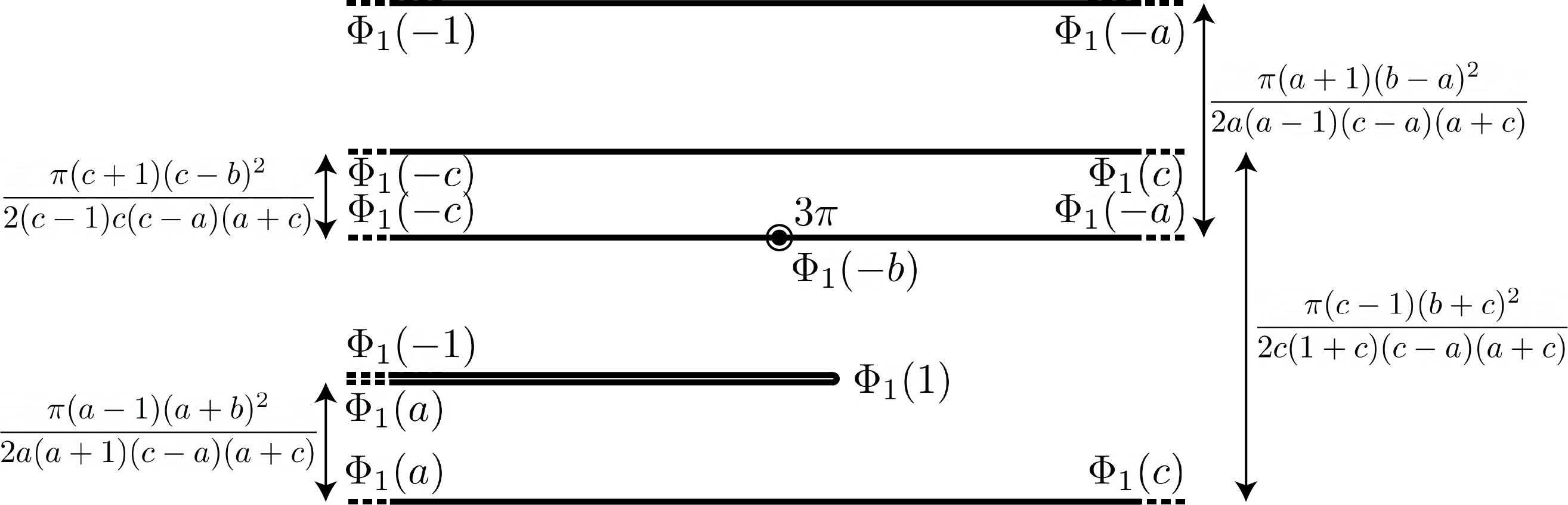}}
		\subfigure[$\frac{1}{G}dh$]{\includegraphics[width=3in]{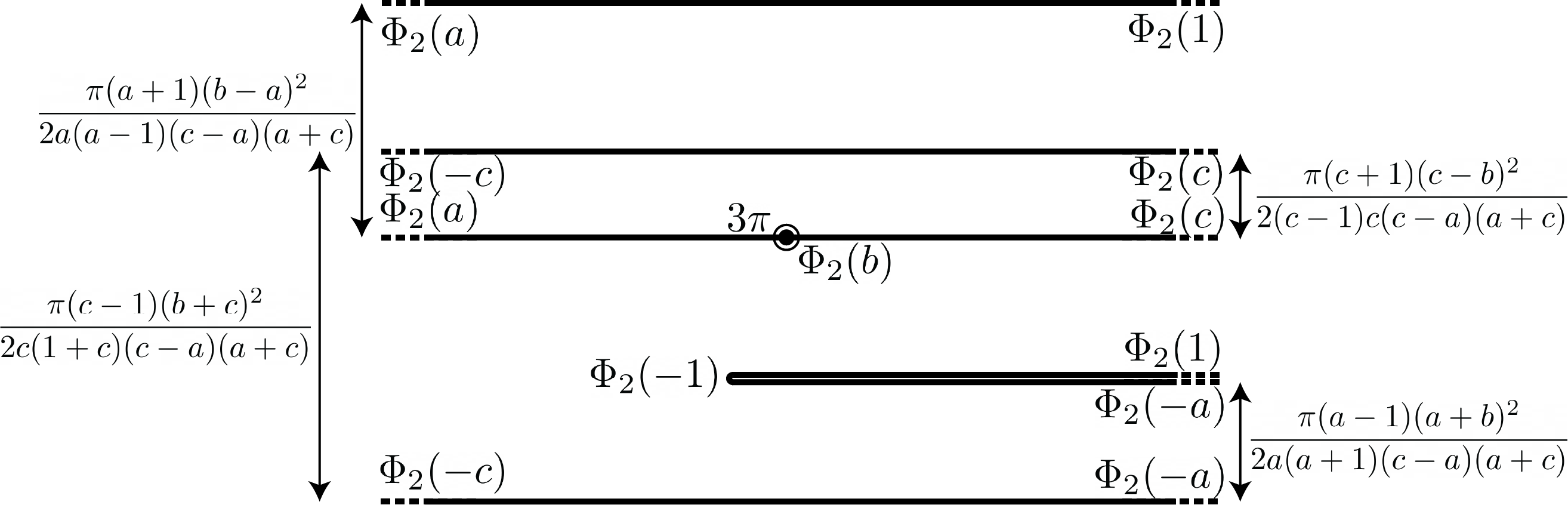}}
	\end{center}
	\caption{$DCCW_n$ flat structures as $\alpha \rightarrow 0$}\label{fig:DCCWlimitflat}	
\end{figure}
Then by an application of Proposition \ref{prop:symmetries}, we finish the construction of the polygons in Figure \ref{fig:DCCWhexagon}.
\end{proof}

\subsection{Period Problem}	\label{DCCWperiodproblem}
 In order to extend the minimal hexagon, obtained in Lemma \ref{lem:DCCWhexagon}, to a minimal wedge $W_{2\pi\alpha}$, we need to show that the sides of the hexagon in Figure \ref{fig:DCCWhexagon}(a) that make the same angle with the $x$-axis, lie on the same plane. This is similar to the extension of the half wedge corresponding to the $DE_{3,n}$ surface found in Section \ref{sec:DEperiod}.
	Thanks to the fact that these hexagons are symmetric with respect to reflections on a horizontal plane, this can be achieved by showing that the two conditions in \eqref{DCCWperiod} are satisfied: 
\begin{equation}\label{DCCWperiod}
	\int_{\gamma_{1,a}}Gdh=\conj{\int_{\gamma_{1,a}}\frac{1}{G}dh}, \qquad 	\int_{\gamma_{a,c}}Gdh=\conj{\int_{\gamma_{a,c}}\frac{1}{G}dh} ,
\end{equation}	
	where $\gamma_{i,j}$ is as described in Section \ref{sec:DEperiod}. Note that for $\alpha=0$, the conditions in \eqref{DCCWperiod} are also valid for the limiting hexagon in Figure \ref{fig:DCCWhexagon}. 

\begin{figure}[H]
	\begin{center}
		\subfigure[$\alpha>0$]{\includegraphics[width=3in]{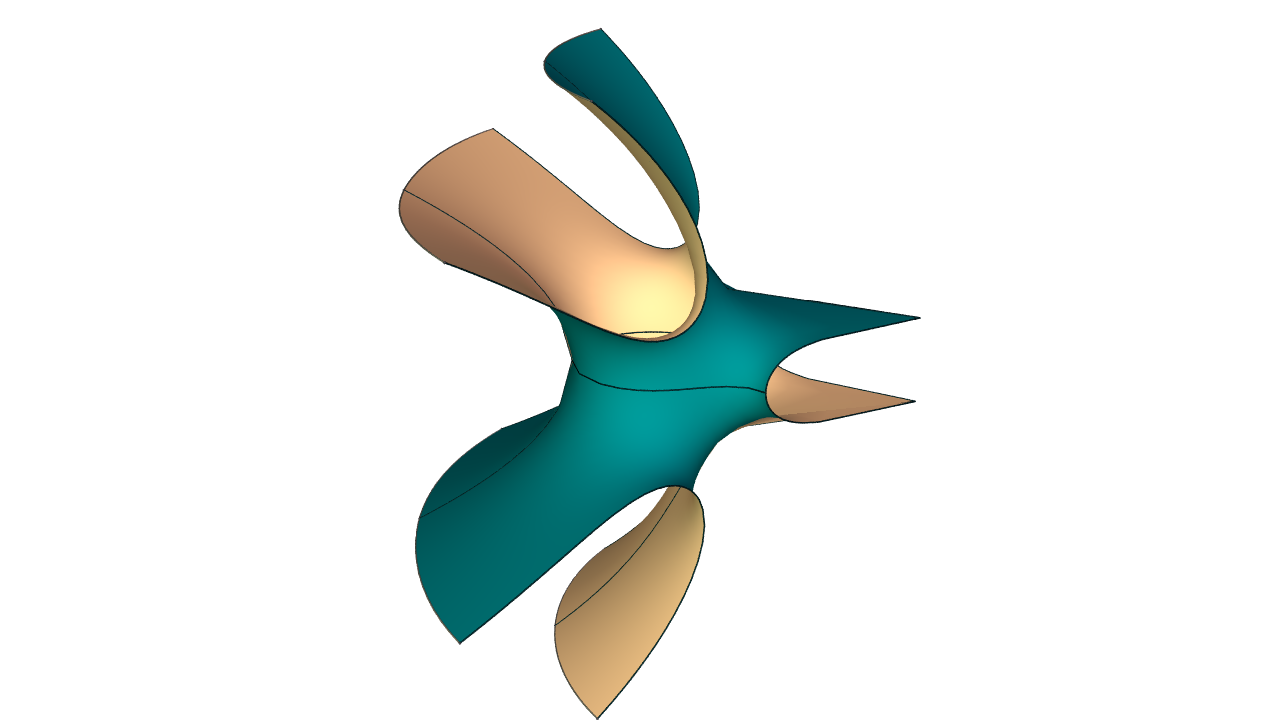}}
		\subfigure[$\alpha=0$]{\includegraphics[width=2.6in]{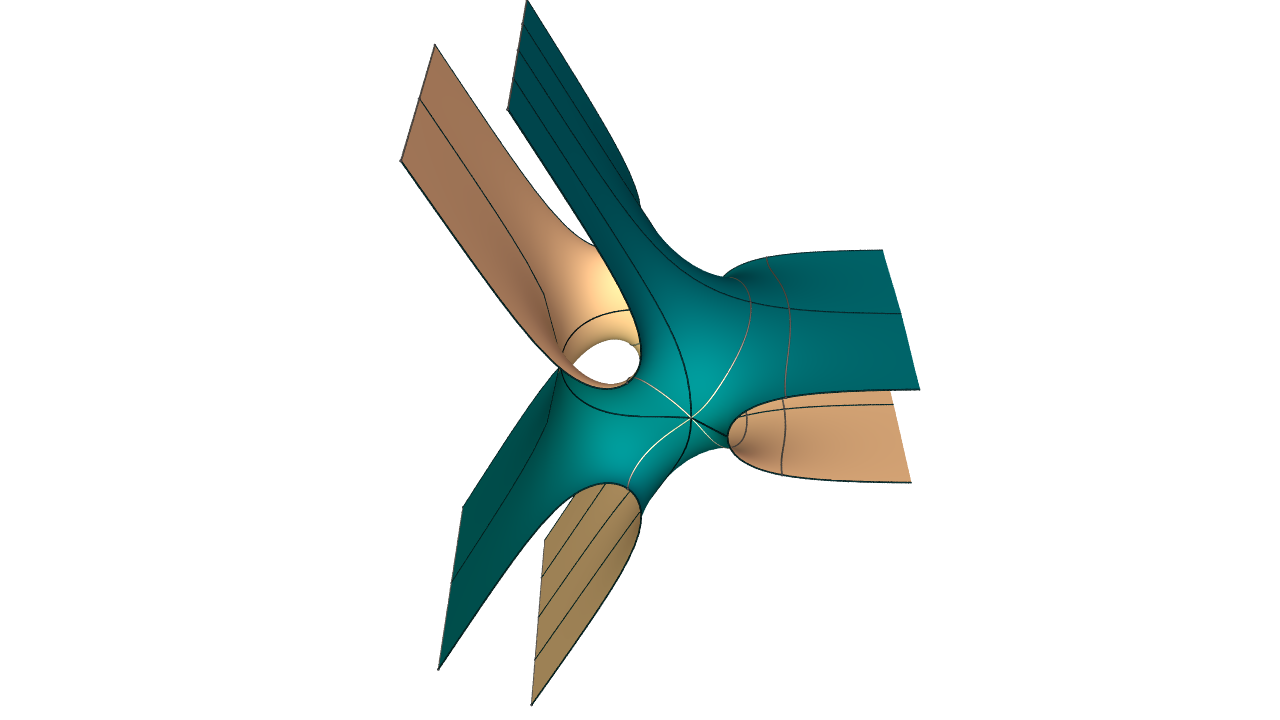}}
	\end{center}
	\caption{$W_{2\pi\alpha}$ corresponding to $DCCW_{n}$}
\end{figure}

	We reformulate the period conditions in \eqref{DCCWperiod} of $W_{2\pi\alpha}$ as follows. First we define the period function $P$ by
	\begin{equation*}
		P(a,b,c,\alpha):=\Bigg\{ 		\int_{\gamma_{1,a}}G dh-\conj{\int_{\gamma_{1,a}}\frac{1}{G}dh}
		,			\int_{\gamma_{a,c}}Gdh-\conj{\int_{\gamma_{a,c}}\frac{1}{G}dh}
		\Bigg\}.
	\end{equation*}
	For a given $\alpha>0$, the existence of a triple $a,b,c$  satisfying the condition $P(a,b,c,\alpha)=0$ implies the existence of a $W_{2\pi\alpha}$ surface, which is a corresponding wedge of the $DCCW_n$ surface. Observe that for $\alpha=0$, using the Residue Theorem, we can write the first coordinate of $P$ as
	\begin{align*}
		\int_{\gamma_{1,a}}Gdh-\conj{\int_{\gamma_{1,a}}\frac{1}{G}dh}&=\pi i\bigg(\Res(Gdh,1)+\Res(Gdh,a)-\conj{\Res(\frac{1}{G}dh,1)}-\conj{\Res(\frac{1}{G}dh,a)}\bigg)\\
		&=\pi i\bigg(\frac{(a+b)^2(1-a)}{(a+1)(c^2-a^2)(2a)}+\frac{2(b-1)^2}{(c^2-1)(a^2-1)}\\
		&\qquad \qquad+\frac{(b-a)^2(a+1)}{(1-a)(c^2-a^2)(2a)}\bigg),
		\shortintertext{whereas the second coordinate of $P$ can be written as:}	
		\int_{\gamma_{a,c}}Gdh-\conj{\int_{\gamma_{a,c}}\frac{1}{G}dh}&=\pi i\bigg(\Res(Gdh,a)+\Res(Gdh,c)-\conj{\Res(\frac{1}{G}dh,a)}-\conj{\Res(\frac{1}{G}dh,c)}\bigg)\\
		&=\pi i\bigg(\frac{(a+b)^2(1-a)}{(a+1)(c^2-a^2)(2a)}+\frac{(c+b)^2(1-c)}{(c+1)(2c)(a^2-c^2)} \\
		&\qquad \qquad+\frac{(b-a)^2(a+1)}{(1-a)(c^2-a^2)(2a)}+\frac{(b-c)^2(c+1)}{(1-c)(2c)(a^2-c^2)}\bigg).
	\end{align*}
Next we regenerate solutions the period problem of $W_{2\pi\alpha}$ corresponding to $DCCC_{n}$. 
\begin{lemma}\label{lem:DCCWwedge}
	For small enough $\alpha \geq 0$, there exists a minimal wedge $W_{2\pi\alpha}$, corresponding to $DCCW_{n}$ given in Theorem \ref{DCCWtheorem}.
\end{lemma}
\begin{proof}
	It is easy to see for $a=-3 \sqrt{2}+3 \sqrt{3}+2 \sqrt{6}-4$, $b=2 \sqrt{2}+3$ and for $\alpha=0$, that $P(a,b,\frac{b^2}{a},0)=0$. One can show that the determinant of the Jacobian of $P(a,b,\frac{b^2}{a},0)$ is precisely equal to $0.000151467$ for values of $a$ and $b$ given above. Here, we exclude writing the explicit expression for the determinant of the Jacobian of $P(a,b,\frac{b^2}{a},0)$ as it is a rational function of polynomials in two variables whose numerator is of degree $11$ and denominator is of degree $16$. Then by the Implicit Function Theorem, for small enough $\alpha \geq 0$, the period condition given in \eqref{DCCWperiod} can be solved. Thus we conclude the proof of Lemma \ref{lem:DCCWwedge}.
\end{proof}
\subsection{Extension, Types of Ends and Growth rates}	
In Section \ref{DCCWperiodproblem}, we demonstrated how the period problem can be solved for a natural number $n$ and for $\alpha=\frac{1}{n}$. Now, a winding number argument similar to the one in Section \ref{sec:type of ends}, shows that the ends of $DCCW_n$ are catenoidal. The growth rate of catenoidal end around $a$ is $\frac{b^2-a^2}{2a(a^2-c^2)}$ and that around $c$ is $\frac{c^2-b^2}{2c(a^2-c^2)}$. For simplicity, we look for solutions where the two growth rates are equal. One can easily show that for $c=\frac{b^2}{a}$, these growth rates are equal. Geometrically this means that outside of a compact cylinder in $\R^3$, $DCCW_n$ looks like a union of two asymptotic concentric catenoids, one inside of the other. On the other hand at $\alpha=0$, the catenoidal ends are replaced by Scherk type ends and the limit surface of $DCCW_n$ is a singly periodic minimal surface with genus $0$. Joaquin Perez and  Martin Traizet proved in \cite{pt1} that Hermann Karcher’s saddle towers are the only singly periodic minimal surfaces of genus 0 with annular ends. In particular, the dihedral limit of $DCCW_n$ is a lesser symmetric Scherk Tower with $6$ ends. 
\begin{remark}\label{rem:DCCWwedge}
	The catenoidal ends of the minimal wedges $W_{2\pi\alpha}$, obtained in Lemma \ref{lem:DCCWwedge}, have equal growth rates.
\end{remark}

We still need to show that the dihedral limit of $DCCW_n$ is a lesser symmetric saddle tower with 6 ends. One can show that for the values of $a,b,c$ given in Lemma \ref{lem:DCCWwedge} i.e. $a=-3 \sqrt{2}+3 \sqrt{3}+2 \sqrt{6}-4$, $b=2 \sqrt{2}+3$ and $c=\frac{b^2}{a}$, the M\"{o}bius transformation that maps $1$ to $c$, $a$ to $-c$ and $c$ to $-a$ is an orientation preserving automorphism of the upper half plane. Additionally, this map is realized by the Weierstrass map $f$, as a 3rd order rotational symmetry of the limit surface, around a line that is parallel to the $y$-axis in $\R^3$. Combining this information with Lemma \ref{lem:DCCWwedge} we finish the proof of Theorem \ref{DCCWtheorem}. 
\end{proof}

\section{Dihedralized Singly Period Karcher Scherk Surface with Handles}\label{sec:DKS}

In this section we will discuss the proof of Theorem \ref{DKStheorem}. Unlike the previous three examples, we are unable to construct $DKS_n$ surfaces using the Schwarz-Christoffel maps because $DKS_n/\langle r \rangle$, where $r$ is $n$th order rotation, is of genus $1$. Instead we will utilize $\vartheta$-functions introduced in Section \ref{sec:geom}.

\begin{proof}[Proof of Theorem \ref{DKStheorem}]
We start the construction of this surface on a rectangular torus $\C/\Lambda_{\tau}$ as follows:
Let $\Gamma$ be the quarter of a fundamental parallelogram of $\Lambda_{\tau}$ with $\partial\Gamma=\partial\Gamma_{dl}\cup\partial\Gamma_{dr}\cup\partial\Gamma_r\cup\partial\Gamma_u\cup\partial\Gamma_{lu}\cup\partial\Gamma_{ld}$ where $\partial\Gamma_{dl}:=\{x: x\in(0,\frac{1}{2}-a) \}$, $\partial\Gamma_{dr}:=\{x: x\in(\frac{1}{2}-a,\frac{1}{2}) \}$, $\partial\Gamma_r:=\{\frac{1}{2}+y \tau: y\in(0,1/2) \}$, $\partial\Gamma_u:=\{s: x+\frac{\tau}{2}: x\in(0,1/2) \}$, $\partial\Gamma_{lu}:=\{y i: y\in(\frac{\im(\tau)}{2}-c,\frac{\im(\tau)}{2})\}$ and $\partial\Gamma_{ld}:=\{y i: y\in(0,\frac{\im(\tau)}{2}-c) \}$, as shown in Figure \ref{fig:quarterdomain}.
	\begin{figure}[H]
	\begin{center}
		\includegraphics[scale=.6]{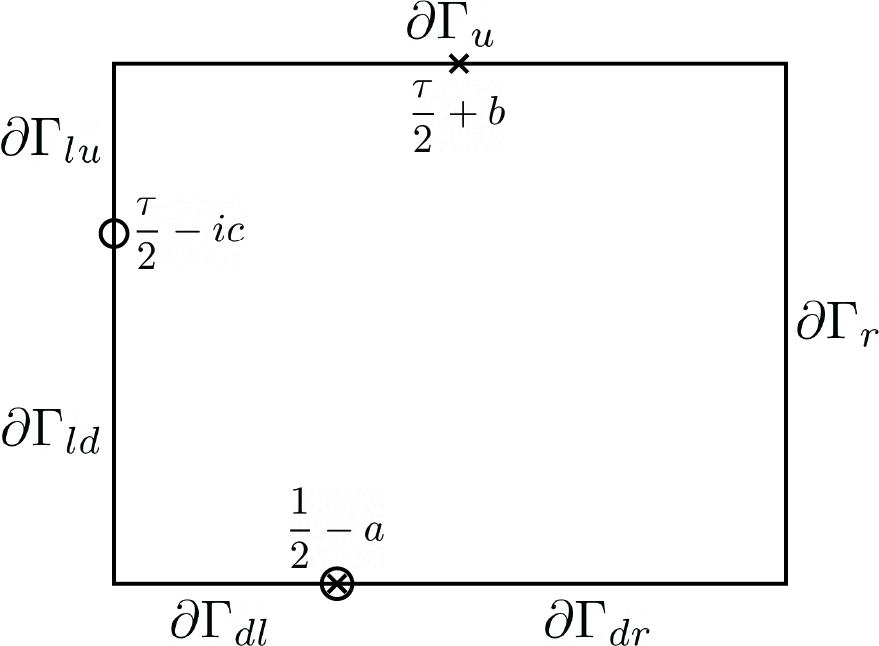}
	\end{center}\caption{A quarter fundamental domain}\label{fig:quarterdomain}
	
\end{figure}

\begin{lemma}\label{lem:DKKShexagon}
	Let $a,b\in(0,\frac{1}{2})$, $\tau\in i\R^+$ and $c\in(0,\im(\frac{\tau}{2}))$. Then there is a Weierstrass map $f$ that maps $\Gamma$ to a minimal hexagon where the images of the horizontal sides  $\partial\Gamma_{dr}$, $\partial\Gamma_{dr}$ and $\partial\Gamma_{u}$ are the edges that lie on vertical symmetry planes and the images of the vertical sides $\partial\Gamma_{r}$, $\partial\Gamma_{lu}$ and $\partial\Gamma_{ld}$ are the edges that lie on horizontal symmetry planes. Moreover, for $\alpha>0$, $f(\frac{\tau}{2}-ic)$ is the only vertex at $\infty$ with an interior angle $0$, $f(\frac{1}{2}-a)$ is a vertex with an interior angle $\alpha\pi$, and the corners of the quarter rectangle are mapped to the vertices where edges in $\R^3$ meet orthogonally. For $\alpha=0$, both $f(\frac{\tau}{2}-ic)$ and $f(\frac{1}{2}-a)$ are vertices at $\infty$.
\end{lemma}

\begin{figure}[H]
	\begin{center}
		\subfigure[$\alpha>0$]{\includegraphics[width=3in]{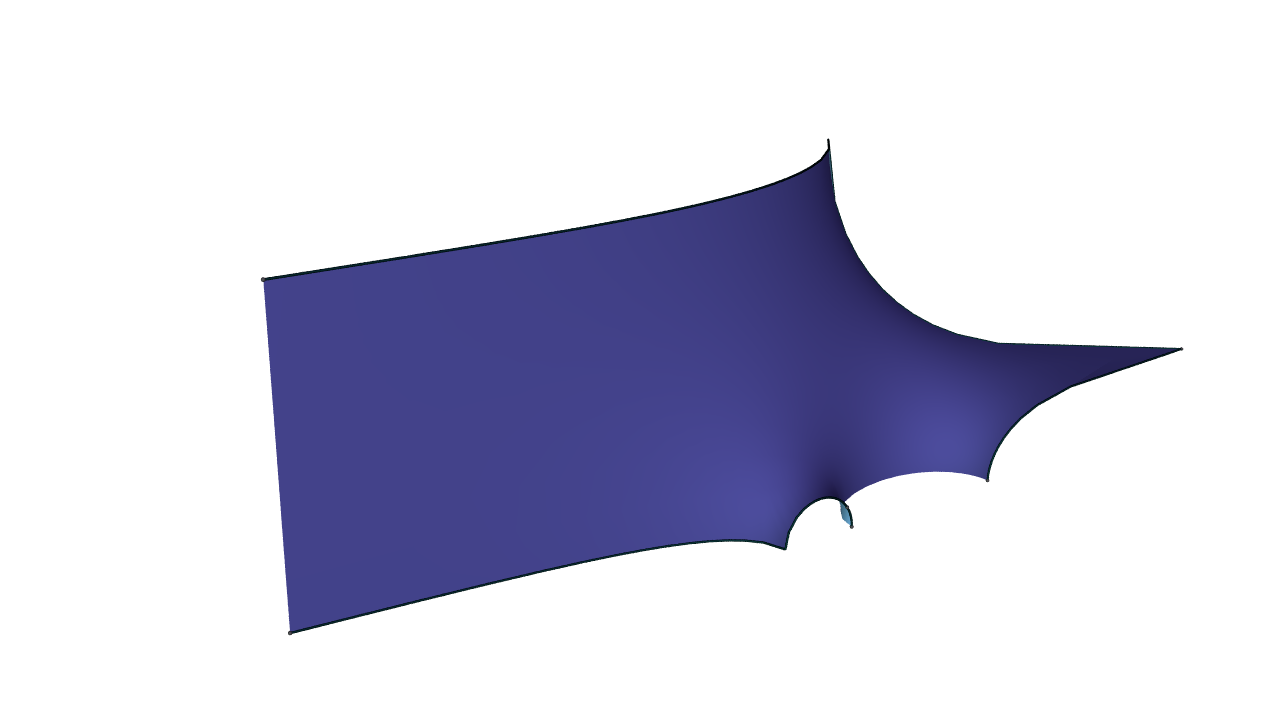}}
		\subfigure[$\alpha=0$]{\includegraphics[width=2.6in]{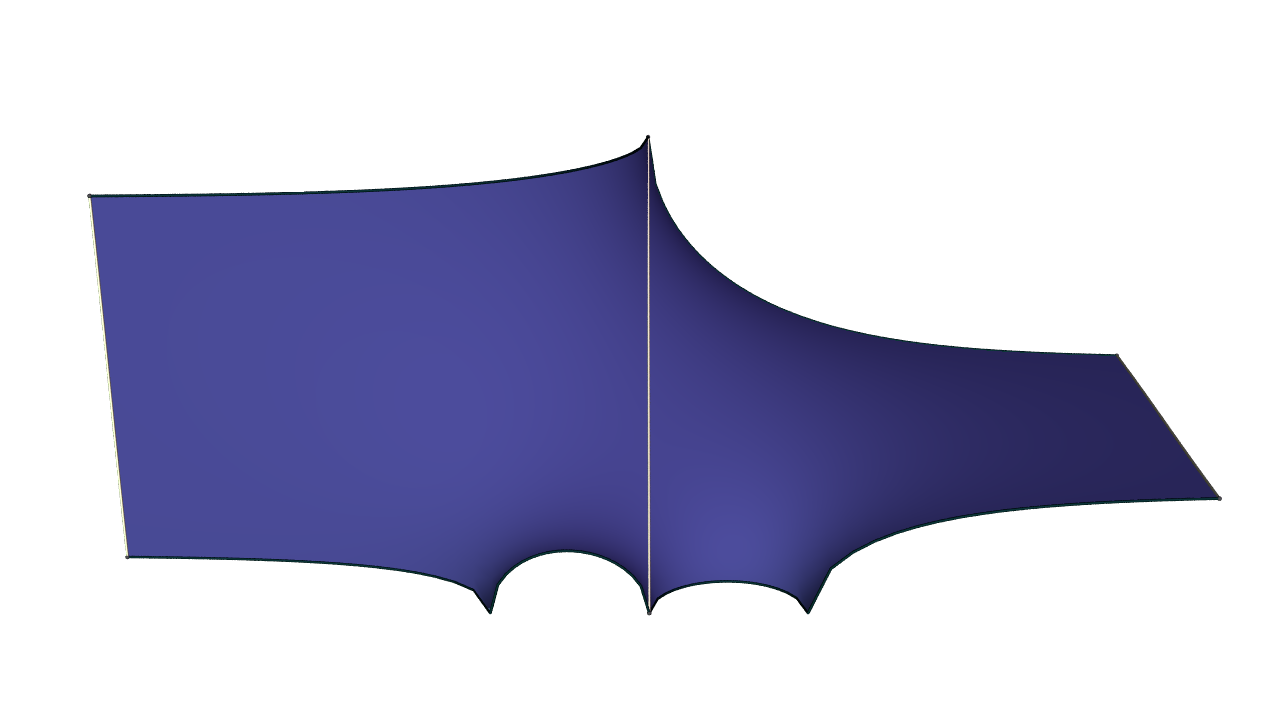}}
	\end{center}
	\caption{Minimal Hexagons in $\R^3$}
	\label{fig:DKShexagon}
\end{figure}
We begin by defining $G$ and $dh$ in the following way:
\begin{align*}
G&=e^{-2\pi i b}\frac{\vartheta(z-(\frac{\tau}{2}-b))}{\vartheta(z-(\frac{\tau}{2}+b))}
	\bigg( \frac{\vartheta(z-(\frac{1}{2}+a))}{\vartheta(z-(\frac{1}{2}-a))}\bigg)^{1-\alpha},
\quad \\
dh&=\frac{\vartheta(z-(\frac{\tau}{2}+b))\vartheta(z-(\frac{\tau}{2}-b))}{\vartheta(z-(\frac{\tau}{2}-ic))\vartheta(z-(\frac{\tau}{2}+ic))} dz .
\end{align*}
Then by using the properties of $\vartheta$-functions we obtain that,
\begin{align*}
	\vartheta(\conj{z}-(\frac{1}{2}+a))&=\vartheta(\conj{z-(\frac{1}{2}+a)})=\conj{\vartheta(z-(\frac{1}{2}+a))},\\
	\vartheta(\conj{z}-(\frac{\tau}{2}-b))&=\vartheta(\conj{z-(-\frac{\tau}{2}-b)})=\conj{\vartheta(z-(-\frac{\tau}{2}-b))}\\
	&=\conj{-e^{-\pi i\tau-2\pi i(z-(\frac{\tau}{2}-b)) }\vartheta(z-(\frac{\tau}{2}-b))} ,\\	
	\vartheta(\conj{z}-(\frac{\tau}{2}+ic))&=\vartheta(\conj{z-(-\frac{\tau}{2}-ic)})=\conj{\vartheta(z-(-\frac{\tau}{2}-ic))}\\
	&=\conj{-e^{-\pi i\tau-2\pi i(z-(\frac{\tau}{2}-ic)) }\vartheta(z-(\frac{\tau}{2}-ic))}.\\	
\end{align*}
By using these transformation formulas, we infer that,
\begin{align*}
	e^{-2\pi i b}\frac{\vartheta(\conj{z}-(\frac{\tau}{2}-b))}{\vartheta(\conj{z}-(\frac{\tau}{2}+b))}&=e^{-2\pi ib}\frac{\conj{-e^{-\pi i\tau-2\pi i(z-(\frac{\tau}{2}-b)) }\vartheta(z-(\frac{\tau}{2}-b))}}{\conj{-e^{-\pi i\tau-2\pi i(z-(\frac{\tau}{2}+b)) }\vartheta(z-(\frac{\tau}{2}+b))}}=\conj{e^{-2\pi ib}\frac{\vartheta(z-(\frac{\tau}{2}-b))}{\vartheta(z-(\frac{\tau}{2}+b))}},
	\end{align*}
	and that,
	\begin{align*}
	&\frac{\vartheta(\conj{z}-(\frac{\tau}{2}+b))\vartheta(\conj{z}-(\frac{\tau}{2}-b))}{\vartheta(\conj{z}-(\frac{\tau}{2}-ic))\vartheta(\conj{z}-(\frac{\tau}{2}+ic))}\\
	&\hspace{0.5in}=	\frac{\conj{e^{-\pi i\tau-2\pi i(z-(\frac{\tau}{2}+b)) }\vartheta(z-(\frac{\tau}{2}+b))} \conj{e^{-\pi i\tau-2\pi i(z-(\frac{\tau}{2}-b)) }\vartheta(z-(\frac{\tau}{2}-b))} }{\conj{e^{-\pi i\tau-2\pi i(z-(\frac{\tau}{2}+ic)) }\vartheta(z-(\frac{\tau}{2}+ic))}\conj{e^{-\pi i\tau-2\pi i(z-(\frac{\tau}{2}-ic)) }\vartheta(z-(\frac{\tau}{2}-ic))}}\\
	&\hspace{0.5in}=	\frac{\conj{\vartheta(z-(\frac{\tau}{2}+b))} \conj{\vartheta(z-(\frac{\tau}{2}-b))} }{\conj{\vartheta(z-(\frac{\tau}{2}+ic))}\conj{\vartheta(z-(\frac{\tau}{2}-ic))}}.
\end{align*}
This implies that $\frac{\vartheta(z-(\frac{1}{2}+a))}{\vartheta(z-(\frac{1}{2}-a))}$, $e^{-2\pi i b}\frac{\vartheta(\conj{z}-(\frac{\tau}{2}-b))}{\vartheta(\conj{z}-(\frac{\tau}{2}+b))}$ and $dh$ assume real values on the real line. That is to say that $\Phi_{1}$ and $\Phi_{2}$ map $\partial\Gamma_{dl}$ and $\partial\Gamma_{dl}$ into straight lines while $dh$ maps the two into real line. Therefore, due to Proposition \ref{prop:symmetries}, we can see that $f(\partial\Gamma_{dl})$ and $f(\partial\Gamma_{dr})$ lie on a vertical symmetry planes. By choosing the branch cut of $1-\alpha$ that makes $\big( \frac{\vartheta(z-(\frac{1}{2}+a))}{\vartheta(z-(\frac{1}{2}-a))}\big)^{1-\alpha}$ real on $[0,\frac{1}{2}-a]$, we obtain that $f(\partial\Gamma_{dl})$ lies on a vertical symmetry plane parallel to the $x$-axis and $f(\partial\Gamma_{dr})$ lies on a vertical symmetry plane that meets the $x$-axis at an angle of $\alpha\pi$.

Similarly, using the properties of $\vartheta$-functions we can additionally show that,
\begin{equation*}
	G(\conj{z}+\tau)=e^{-4\pi i b}(e^{4\pi ia})^{1-\alpha}\conj{G(z)}, \qquad dh(\conj{z}+\tau)=\conj{dh(z)}.
\end{equation*}
This tells us that on $\partial\Gamma_{u}$, $dh$ is real and the argument of $G$ is $-2\pi b+2\pi a(1-\alpha)$ modulo $2\pi (1-\alpha)$. Hence,  $f(\partial\Gamma_{u})$ lies on a vertical symmetry plane. In the following Proposition, we determine the exact angle between the plane containing $f(\partial\Gamma_{u})$ and the plane containing $f(\partial\Gamma_{dr})$.
\begin{proposition}\label{prop:parallelplanes}
	For $b=a(1-\alpha)+\frac{\alpha}{2}$, the two planes that contain images of $f(\partial\Gamma_{u})$ and $f(\partial\Gamma_{dr})$ are parallel.
\end{proposition}
\begin{proof}
	Note that for any complex number $x$, we have 
\begin{equation}\label{thetatrans}
\begin{split}
\vartheta(\frac{\tau}{2}+x)&=\vartheta(\tau-\frac{\tau}{2}+x) =-e^{-2\pi i x}\vartheta(-\frac{\tau}{2}+x)=e^{-2\pi i x}\vartheta(\frac{\tau}{2}-x),\\
 \vartheta(\frac{1}{2}-x)&=\vartheta(1-\frac{1}{2}-x) =-\vartheta(-\frac{1}{2}-x) =\vartheta(\frac{1}{2}+x),\\
 \vartheta(\frac{1}{2}-(\frac{\tau}{2}-x))&=e^{2\pi i x}\vartheta(\frac{1}{2}-(\frac{\tau}{2}-\tau-x))=e^{2\pi i x}\vartheta(\frac{1}{2}-(-\frac{\tau}{2}-x))\\ &
 =e^{2\pi  x}\vartheta(\frac{1}{2}-(\frac{\tau}{2}+x)).
\end{split}
\end{equation}

First, we show that the following holds true:
\begin{align}\label{identity}
	\arg\Bigg(\bigg(\frac{\vartheta(\frac{\tau}{2}-a)}{\vartheta(\frac{\tau}{2}+a)}\bigg)^{1-\alpha}\Bigg/ \bigg(\frac{\vartheta(-a)}{\vartheta(a)}\bigg)^{1-\alpha}\Bigg)=2\pi a(1-\alpha)- \pi(1-\alpha).
\end{align}
	Equation \eqref{identity} being modulo $2\pi(1-\alpha)$ is an immediate consequence of \eqref{thetatrans}. Observe that both sides of the equation are continuous in $a$ for $a>0$. In particular, when $a=1/2$, the expression appearing inside the argument on the left-hand side of \eqref{identity} is equal to $1$;  and thus by continuity the right-hand of \eqref{identity} is equal to $0$. Hence we have proven \eqref{identity}.
	
	 Next, consider the change in argument of the Gauss map $G$:
\begin{align*}
	&\arg(G(\frac{\tau}{2}+\frac{1}{2}))-\arg(G(\frac{1}{2}))\\
	&\hspace{0.5in}=\arg\Bigg(\frac{\vartheta(\frac{1}{2}+b)}{\vartheta(\frac{1}{2}-b)}\bigg(\frac{\vartheta(\frac{\tau}{2}-a)}{\vartheta(\frac{\tau}{2}+a)}\bigg)^{1-\alpha}\Bigg/ \frac{\vartheta(\frac{1}{2}-(\frac{\tau}{2}-b))}{\vartheta(\frac{1}{2}-(\frac{\tau}{2}+b))}\bigg(\frac{\vartheta(-a)}{\vartheta(a)}\bigg)^{1-\alpha}\Bigg)\\
	&\hspace{0.5in}=\arg\Bigg(\frac{\vartheta(\frac{1}{2}+b)}{\vartheta(\frac{1}{2}-b)}\Bigg/ \frac{\vartheta(\frac{1}{2}-(\frac{\tau}{2}-b))}{\vartheta(\frac{1}{2}-(\frac{\tau}{2}+b))}\Bigg) \\
	&\hspace{0.75in}+\arg\Bigg( \bigg(\frac{\vartheta(\frac{\tau}{2}-a)}{\vartheta(\frac{\tau}{2}+a)}\bigg)^{1-\alpha}\Bigg/\bigg(\frac{\vartheta(-a)}{\vartheta(a)}\bigg)^{1-\alpha}\Bigg)\\
	&\hspace{0.5in}=-2\pi b+2\pi a(1-\alpha)- \pi(1-\alpha).
\end{align*}
Note that for $b=a(1-\alpha)+\frac{\alpha}{2}$, we have $-2\pi b+2\pi a(1-\alpha)- \pi(1-\alpha)=-\pi$. Since we know that $\frac{\tau}{2}+\frac{1}{2}\in\partial\Gamma_{u}$, $\frac{1}{2}\in\partial\Gamma_{dr}$ and that $f(\partial\Gamma_{u})$ and $f(\partial\Gamma_{dr})$ lie on vertical symmetry planes, Proposition \ref{prop:parallelplanes} follows.
\end{proof}

Hereafter, we will fix $b=a(1-\alpha)+\frac{\alpha}{2}$.

We will now show that the vertical sides of $\partial\Gamma$ are mapped to the edges in horizontal symmetry planes. Observe that at $0\in\partial\Gamma_{ld}$, the Gauss map $G$ attains the value $1$, that is, we do not need to normalize $G$.  Now, let $\sigma$ be the reflection on the imaginary axis bounding $\Gamma$. Then, we can extend the Weiestrass data defined on $\Gamma$ over the imaginary line i.e. extend the surface and obtain:
\begin{align*}
	\sigma^{*}G &=e^{-2\pi b}\frac{\vartheta(-\conj{z}-(\frac{\tau}{2}-b))}{\vartheta(-\conj{z}-(\frac{\tau}{2}+b))}
	\bigg( \frac{\vartheta(-\conj{z}-(\frac{1}{2}+a))}{\vartheta(-\conj{z}-(\frac{1}{2}-a))}\bigg)^{1-\alpha}\\
	&=e^{-2\pi b}\frac{\conj{\vartheta(z-(\frac{\tau}{2}+b))}}{\conj{\vartheta(z-(\frac{\tau}{2}-b))}}
	\bigg( \frac{\conj{\vartheta(z-(\frac{1}{2}-a))}}{\conj{\vartheta(z-(\frac{1}{2}+a))}}\bigg)^{1-\alpha}\\ &=\conj{\frac{1}{G}},
\end{align*}
and
\begin{align*}
	\sigma^{*}dh&=\frac{\vartheta(-\conj{z}-(\frac{\tau}{2}+b))\vartheta(-\conj{z}-(\frac{\tau}{2}-b))}{\vartheta(-\conj{z}-(\frac{\tau}{2}-ic))\vartheta(-\conj{z}-(\frac{\tau}{2}+ic))}\sigma^*dz\\
	&=-\frac{\conj{\vartheta(z-(\frac{\tau}{2}+b))\vartheta(z-(\frac{\tau}{2}-b))}}{\conj{\vartheta(z-(\frac{\tau}{2}-ic))\vartheta(z-(\frac{\tau}{2}+ic))}}\conj{dz}\\
	&=-\conj{dh}.
\end{align*}
Hence, same argument that we made for $DCCW_n$ shows that, by Proposition \ref{prop:symmetries}, $\partial\Gamma_{ld}$ and $\partial\Gamma_{lu}$ are mapped into horizontal symmetry planes. Almost identical calculations show that $\partial\Gamma_{r}$ is in a horizontal symmetry plane as well.

Combining all the information we deduce about the image of $\Gamma$ under the Weierstrass map $f$, we deduce that $f(\Gamma)$ is indeed a minimal hexagon (see Figure \ref{fig:DKShexagon}).
\subsection{Period Problem of $DKS_n$}

In order to close the periods, we need to make sure that the two horizontal planes planes containing the images of $\partial\Gamma_{lu}$ and $\partial\Gamma_{r}$ are of the same height and that the images of $\partial\Gamma_{dr}$ and $\partial\Gamma_{u}$ are on the same plane. The former condition can be achieved by closing the vertical period condition along $\partial\Gamma_{u}$, whereas the latter can be achieved by closing the horizontal period condition along  $\partial\Gamma_{r}$. Note that by horizontal and vertical we refer to the their trace in $\R^3$ rather than their position on the torus. That is we need to show that the following is true:
\begin{equation}\label{eq=DKSperiod}
	\int_{\partial\Gamma_{u}}dh=0 \qquad 	\int_{\partial\Gamma_{r}}Gdh=\conj{\int_{\partial\Gamma_{r}}\frac{1}{G}dh} \qquad 
\end{equation}
For a given dihedral angle $\alpha$, we thus need to solve for $a,c,\tau$ such that the following function is equal to $0$:
\begin{equation}\label{eq=DKSperiodf}
P(a,c,\tau,\alpha):=\Bigg\{\int_{\partial\Gamma_{u}}dh,\int_{\partial\Gamma_{r}}Gdh-\conj{\int_{\partial\Gamma_{r}}\frac{1}{G}dh}\Bigg\}
\end{equation}

An important point to note here is that for $\alpha=0$, $\tau=i$ there exists $a_0$, such that $a=a_0=c$ solves the period condition given in \eqref{eq=DKSperiod}. In fact, in this particular case the Weierstrass data given above parametrizes the genus 1 doubly periodic Scherk Surface, the existence of which was given by  Hermann Karcher in \cite{howe3} using a different parametrization. In this paper, we apply the Implicit Function Theorem and extend this solution to the $DKS_n$ surfaces given in Theorem \ref{DKStheorem}. Since the construction of the period problem is done by using $\vartheta$-functions on varying tori, it can be difficult to apply the Implicit Function Theorem. In fact, in order to show that the Jacobian has non-zero derivative, one cannot use the residue theorem, unlike in the previous examples $DE_{3,n}$, $DH_{1,1,n}$ and $DCCW_n$. Instead, one has to work with the derivatives of the $\vartheta$-functions. However, it is sufficient to calculate the Jacobian for the fix values $\alpha=0$ and $\tau=i$. Then if the Jacobian is nonzero at $a=a_0=c$, we will obtain a two parameter family of solutions to $P$ that varies in $\alpha,\tau$ around $\alpha=0$, $\tau=i$ and $a=a_0=c$. Assuming $\alpha=0$ and $\tau=i$ simplifies the calculations greatly because we can, in fact, rotate the surface in $\R^3$ and reparametrize a simply connected piece of the surface when $\alpha=0$. This can be understood as reparametrizing $\Gamma$ on the upper half plane.
\subsection{Reparametrization of the Limit}
Note that, when $\alpha=0$, $G$ and $dh$ are well-defined on $\C/\Lambda_{\tau}$. We can rotate the limiting hexagon, which appears in Lemma \ref{lem:DKKShexagon}, in $\R^3$ so that all its symmetry planes become vertical symmetry planes. This can be done only when $\alpha=0$, because when $\alpha$ is nonzero, the vertical symmetry planes of the minimal hexagon are not parallel.  Rotating the minimal hexagon changes its Weierstrass data, however, it also allows us to reparametrize the new Weierstrass data on $\Gamma$ by using Schwarz-Christoffel maps on the upper half plane. Let $\tilde{G}$ and $\tilde{dh}$ be the new Weierstrass data of the limiting hexagon, obtained after the abovementioned rotation, whose divisors can be described as follows: Since $G$ is an elliptic function of order 2 on a torus, it assumes every value exactly twice. This implies that $\tilde{G}$ assumes every value on the torus exactly twice as well. In particular, $\tilde{G}$ is vertical at the corner points of $\Gamma$. Since $\tilde G$ can not assume $0$ or $\infty$ more than twice, there is no other point where $\tilde{G}$ is vertical. Consequently, $\tilde{dh}$ has simple $0$'s only at the corners points of $\Gamma$ and at the end points $\frac{\tau}{2}-ic$, $\frac{1}{2}-a$. Now consider the the following map, defining a change of coordinates:
\begin{equation}
	T(z)=\int _{0}^{z}\frac{dw}{\sqrt{w (1-w^2)}}\Bigg/2\int _{0}^{1}\frac{dw}{\sqrt{w (1-w^2)}}.
\end{equation} 
Note that $T$ is a normalized Schwarz-Christoffel map that maps the upper half plane to $\Gamma$ for $\tau=i$. In particular, it maps the points $-1,0,1,\infty$  to $\frac{i}{2},0,\frac{1}{2},\frac{i+1}{2}$, respectively, which are the corners points of $\Gamma$. Thus $\tilde{G}{dh}$ and $\frac{1}{\tilde{G}}\tilde{dh}$ can now be reparametrized on $\mathbb{H}$ by the following Schwarz-Christoffel integrands:
\begin{equation*}
	\tilde{G}\tilde{dh}=\rho\frac{\sqrt{1-z^2}}{\sqrt{z}(z - \tilde{a}) (z + \tilde{c})}, \quad \frac{1}{\tilde{G}}\tilde{dh}=\frac{\sqrt{z}}{\rho\sqrt{1-z^2}(z - \tilde{a}) (z + \tilde{c})},
\end{equation*}
where $\tilde{a}=T^{-1}(\frac{1}{2}-a)$, $\tilde{c}=-T^{-1}(\frac{\tau}{2}-ic)$ and $\rho=\frac{\sqrt{\tilde{a}}}{\sqrt{1-\tilde{a}^2}}$.\\
 Observe that we have $0<\tilde{a},\tilde{c}<1$ using the definition of $T$. 
This further implies that,
\begin{equation*}
	\tilde{G}=\rho\frac{\sqrt{1-z^2}}{\sqrt{z}}, \quad \tilde{dh}=\frac{1}{(z-\tilde{a}) (z+\tilde{c})}.
\end{equation*}
One can show that by pushing forward $\tilde{G}$ and $\tilde{dh}$ to $\Gamma$, by the map $T$, and extending it to $\C/\Lambda_{\tau}$, we obtain the divisors described above. Moreover, since the Gauss map is vertical at $\frac{1}{2}-a$ before rotation, after the rotation it becomes horizontal. Thanks to our choice of $\rho$, we see that $\tilde{G}$ assumes the value $1$ at $\tilde{a}$. Thus we confirm that $\tilde{G}$ and $\tilde{dh}$ are scaled correctly and that they are indeed the reparametrized Weierstrass data on the upper half plane. As a result, we can rewrite the period problem given by equation \eqref{eq=DKSperiodf} as:
\begin{equation}\label{tildeP}
	P(a,c,i,0)=\tilde{P}(\tilde{a},\tilde{c})=\Bigg\{\int_{\infty}^{-1}\tilde{G}\tilde{dh}-\conj{\int_{\infty}^{-1}\frac{1}{\tilde{G}}\tilde{dh}},\int_{1}^{\infty}\tilde{G}\tilde{dh}-\conj{\int_{1}^{\infty}\frac{1}{\tilde{G}}\tilde{dh}}\Bigg\}
\end{equation}
\subsection{Solution to the Period Problem}
Clearly, the integrals appearing in the first coordinate of the definition of $\tilde P$ in \eqref{tildeP}, are real, whereas the integrals in the second coordinate are imaginary. However, we can write both in terms of real valued integrals in the following manner:
\begin{align*}
\int_{\infty}^{-1}\tilde{G}\tilde{dh}-\conj{\int_{\infty}^{-1}\frac{1}{\tilde{G}}\tilde{dh}}&=\int_{\infty}^{-1}\rho\frac{\sqrt{1-z^2}dz}{\sqrt{z}(z - \tilde{a}) (z + \tilde{c})}-\frac{\sqrt{z}dz}{\rho\sqrt{1-z^2}(z - \tilde{a}) (z + \tilde{c})}\\
&=\int_{1}^{\infty}\rho\frac{\sqrt{t^2-1}dt}{\sqrt{t}(t+ \tilde{a}) (t - \tilde{c})}-\frac{\sqrt{t}dt}{\rho\sqrt{t^2-1}(t + \tilde{a}) (t - \tilde{c})}
\shortintertext{and,}
\int_{1}^{\infty}\tilde{G}\tilde{dh}-\conj{\int_{1}^{\infty}\frac{1}{\tilde{G}}\tilde{dh}}&=\int_{1}^{\infty}\rho\frac{\sqrt{1-z^2}dz}{\sqrt{z}(z - \tilde{a}) (z + \tilde{c})}+\frac{\sqrt{z}dz}{\rho\sqrt{1-z^2}(z - \tilde{a}) (z + \tilde{c})}\\
&=i\int_{1}^{\infty}\rho\frac{\sqrt{t^2-1}dt}{\sqrt{t}(t - \tilde{a}) (t + \tilde{c})}-\frac{\sqrt{t}dt}{\rho\sqrt{t^2-1}(t - \tilde{a}) (t + \tilde{c})}.
\end{align*}
For simplicity, we define the following functions:
\begin{equation*}
	\psi_1(x,y):=\int_{1}^{\infty}\frac{\sqrt{t^2-1}dt}{\sqrt{t}(t+x)(t-y)}, \qquad \psi_2(x,y):=\int_{1}^{\infty}\frac{\sqrt{t}dt}{\sqrt{t^2-1}(t+x)(t-y)}.
\end{equation*}
 Now using these new functions we rewrite $\tilde P$ defined in \eqref{tildeP} as,
\begin{align*}
	\tilde{P}(x,y)=\bigg\{\rho(x)\psi_1(x,y)-\frac{\psi_2(x,y)}{\rho(x)},i\rho(x)\psi_1(-x,-y)-i\frac{\psi_2(-x,-y)}{\rho(x)}\bigg\},
\end{align*}
and we calculate the Jacobian of $\tilde P$ to obtain,
\begin{align*}
	-i\det(D(\tilde{P}(x,y)))&=\partial_{x}\bigg(\rho(x)\psi_1(x,y)-\frac{\psi_2(x,y)}{\rho(x)}\bigg)\partial_{y}\bigg(\rho(x)\psi_1(-x,-y)-\frac{\psi_2(-x,-y)}{\rho(x)}\bigg)\\
	&-\partial_{x}\bigg(\rho(x)\psi_1(-x,-y)-\frac{\psi_2(-x,-y)}{\rho(x)}\bigg)\partial_{y}\bigg(\rho(x)\psi_1(x,y)-\frac{\psi_2(x,y)}{\rho(x)}\bigg).\\
	\end{align*}
Note that $	\psi_i(x,x)=\psi_i(-x,-x)$ and $	\partial_{1}\psi_i(x,x)=-\partial_{2}\psi_i(-x,-x)$. Evaluating the Jacobian at $x,y=\tilde{a}$ gives us:
\begin{equation*}
	\det(D(\tilde{P}(\tilde{a},\tilde{a})))=-\frac{i}{(\rho(\tilde{a})^3)}f_1(\tilde{a})f_2(\tilde{a}).
\end{equation*}
Here $f_i$'s are defined as,

\begin{equation*}\label{eq:firstfactor}
	f_1(\tilde{a}):=	(\rho(\tilde{a}))^2\big(\partial_2\psi_1(\tilde{a},\tilde{a})-\partial_1\psi_1(\tilde{a},\tilde{a})\big)-\partial_2\psi_2(\tilde{a},\tilde{a})+\partial_1\psi_2(\tilde{a},\tilde{a}),
\end{equation*}
\begin{align*}\label{eq:secondfactor}
	f_2(\tilde{a})&:=(\rho(\tilde{a}))^2\rho'(\tilde{a})\psi_1(\tilde{a},\tilde{a})+\rho'(\tilde{a})\psi_2(\tilde{a},\tilde{a})+(\rho(\tilde{a}))^3\big(\partial_{2}\psi_1(\tilde{a},\tilde{a})+\partial_{1}\psi_1(\tilde{a},\tilde{a})\big)\\
	&-\rho(\tilde{a})\big(\partial_{2}\psi_2(\tilde{a},\tilde{a})+\partial_{1}\psi_2(\tilde{a},\tilde{a})\big).
\end{align*}
Furthermore, using the definitions of $\psi_2$, $\psi_2$ and $\rho$, we simplify $f_i$'s as,
\begin{equation*}
	f_1(\tilde{a})=\frac{2}{\tilde{a}^2-1}\int_{1}^{\infty}\frac{ \sqrt{t} (1-t\tilde{a})}{\sqrt{t^2-1}(t-\tilde{a})^2 (t+\tilde{a})}dt,
\end{equation*}
\begin{equation*}
	f_2(\tilde{a})=\int_{1}^{\infty}\frac{t^2 \left(\tilde{a}^3+\tilde{a}\right)+t \left(-3 \tilde{a}^4-2 \tilde{a}^2+1\right)+\tilde{a} \left(5 \tilde{a}^2-3\right)}{2 \sqrt{t} \sqrt{t^2-1} \sqrt{\tilde{a}} \left(1-\tilde{a}^2\right)^{5/2} (t-\tilde{a})^2}dt.
\end{equation*}
Observe that, the determinant of the Jacobian of $\tilde{P}(\tilde{a},\tilde{c})$ at $\tilde{a}=\tilde{c}$ is nonzero if and only if $f_1(\tilde{a})$ and $f_2(\tilde{a})$ are nonzero. Numerical results show that, $f_1(\tilde{a})$ is strictly positive and that $f_2(\tilde{a})$ is strictly negative for $\tilde{a}\in(0,1)$. Therefore the determinant of the Jacobian is nonzero for all values of $\tilde{a}\in(0,1)$. In particular, the determinant is nonzero for $\tilde{a}=a_0$, $a_0$ being the solution to the period problem of the genus one doubly periodic Scherk Surface.
\begin{figure}[H]
	\begin{center}
		\includegraphics[width=11cm]{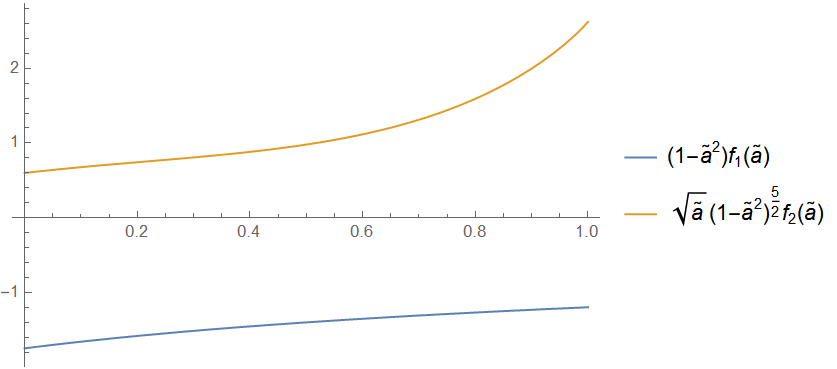}
	\end{center}
\end{figure}
\begin{lemma}
		For every small enough $\alpha\geq 0$ and for $\tau$ in an imaginary interval around around $i$, there exists a minimal wedge $W_{2\pi\alpha}$ corresponding to the $DKS_{n}$ surface given in Theorem \ref{DKStheorem}.
\end{lemma}
\begin{proof}
	Observe that,
	 $$P(a,c,i,0)=\tilde{P}(T^{-1}(\frac{1}{2}-a),-T^{-1}(\frac{i}{2}-ic)).$$
	
	  We have showed that Jacobian of $\tilde{P}(\tilde{a},\tilde{c})$ is nonzero when it is evaluated at $\tilde{a}=\tilde{c}$. Observe also that the Jacobian of the map $(x,y)\mapsto (T^{-1}(\frac{1}{2}-x),-T^{-1}(\frac{i}{2}-iy))$ is nonzero for positive $x,y$. Leveraging these findings, along with the fact that the period problem at limit can be solved for values $\alpha=0$, $\tau=i$, $a=a_0=c$, and employing the Implicit Function Theorem, we conclude that for a given pair $\alpha,\tau$ in a neighborhood around $\alpha=0,\tau=i$, there exist $a \in (0,\frac{1}{2})$ and $c\in (0,\frac{\im(\tau)}{2})$ that satisfy the period conditions given in \eqref{eq=DKSperiod}.
\end{proof}
\begin{figure}[H]
	\begin{center}
		\subfigure[$\alpha>0$]{\includegraphics[width=3in]{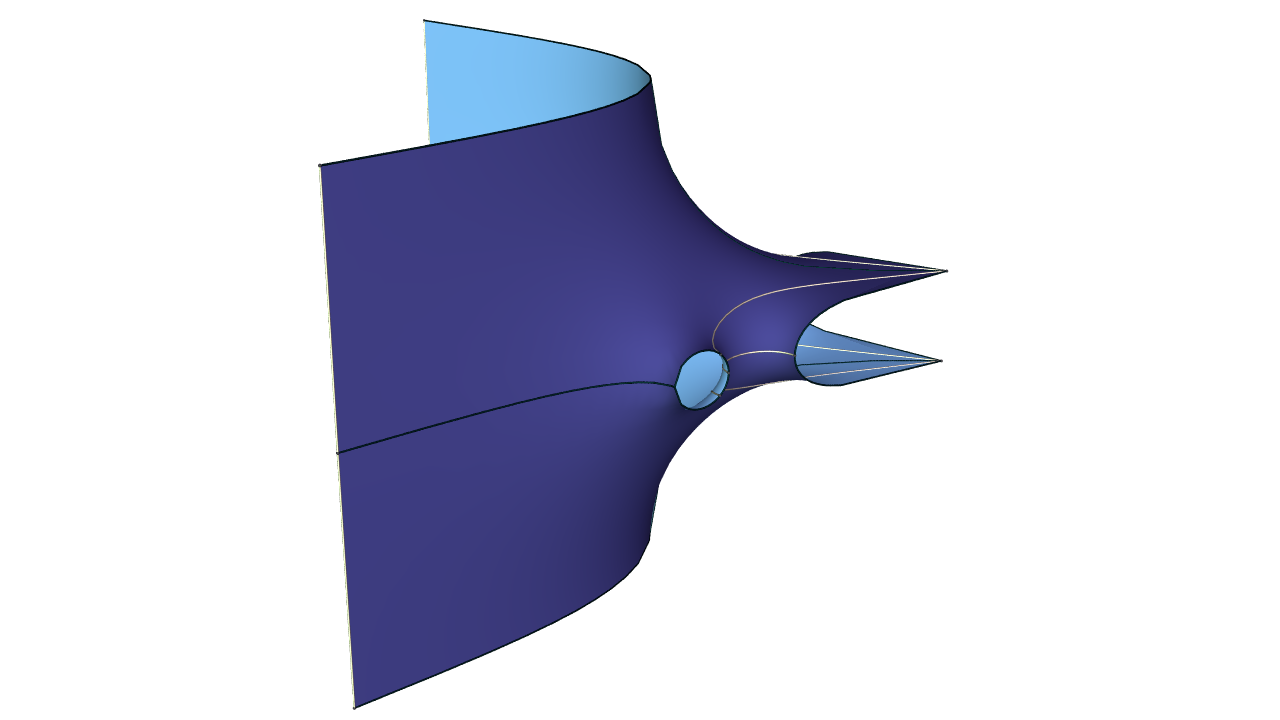}}
		\subfigure[$\alpha=0$]{\includegraphics[width=2.6in]{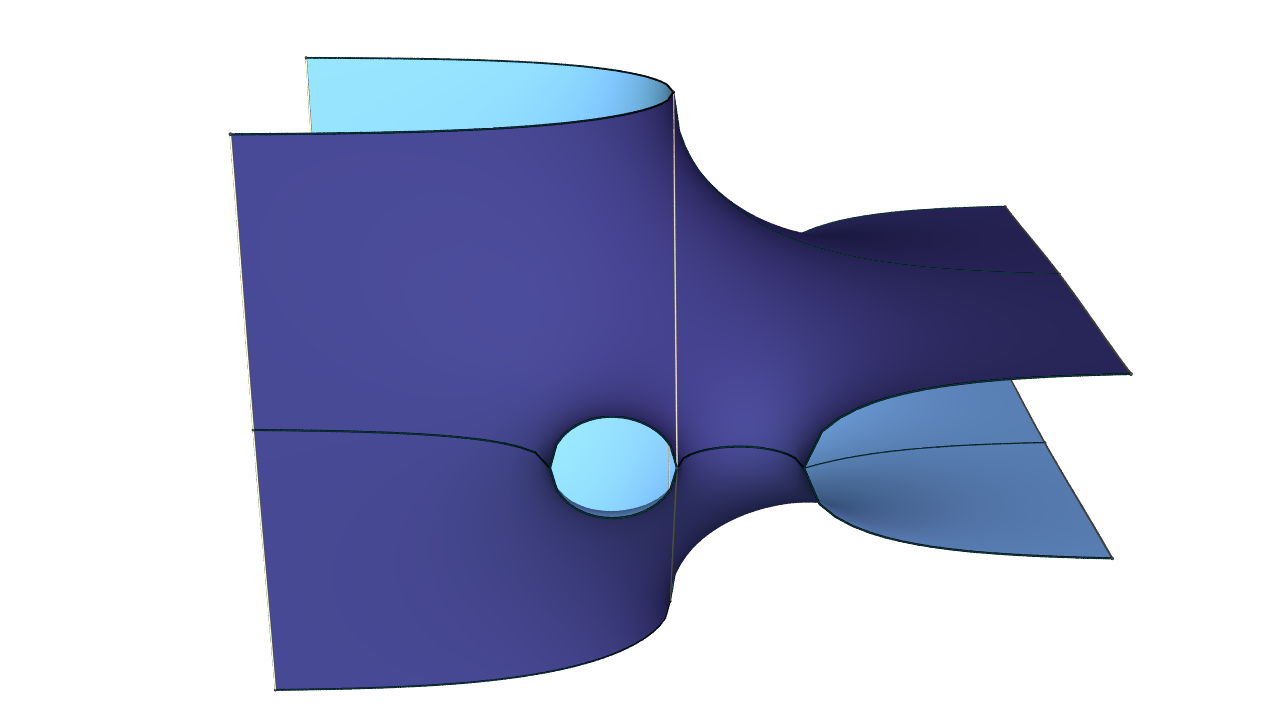}}
	\end{center}
	\caption{$W_{2\pi\alpha}$ corresponding to $DKS_{n}$}
\end{figure}
Thus we conclude the proof of Theorem \ref{DKStheorem}.
\end{proof}
\begin{remark}
	Here we point out the differences in the approaches taken in finding $DCCW_{n}$ and $DKS_{n}$: The sequence $DCCW_{n}$ is first constructed as a variation of Wolgemuth's surface and then its dihedral limit is investigated. On the contrary, the $DKS_{n}$'s are found as a sequence converging to the doubly periodic Karcher-Scherk surface of genus $1$ (see \cite{howe3}), that is, in this scenario the limit candidate is fixed which then gives rise to a sequence of minimal surfaces. Thanks to the two way implementations of the dihedralization method, highlighted in the examples above, we build a bridge between finite type minimal surfaces and singly periodic minimal surfaces, and similarly between singly and doubly periodic minimal surfaces.
\end{remark}

\section{Acknowledgments}
The author would like to thank Professor Matthias Weber for his invaluable help and insights throughout the project.
\bibliography{minlit}
\bibliographystyle{alpha}
\end{document}